\documentclass[12pt,a4paper,intlimits,sumlimits]{amsart}

\usepackage{amsmath, amsthm, mathtools}
\usepackage{amsfonts}
\usepackage{hyperref}
\usepackage[alphabetic, initials, nobysame]{amsrefs}
\usepackage{graphicx}
\usepackage{framed}
\usepackage{amssymb}
\usepackage{esvect}
\usepackage{xcolor}
\usepackage{eucal}
\usepackage[english]{babel}
\usepackage{mathrsfs}
\usepackage{esint}

\def \N {\mathbb{N}}
\def \R {\mathbb{R}}

\def \Ns {\mathscr{N}_{s}}
\def \Q {\mathcal Q}
\allowdisplaybreaks

\theoremstyle{definition}
\newtheorem{definition}{Definition}[section]

\theoremstyle{plain}
\newtheorem{theorem}[definition]{Theorem}
\newtheorem{proposition}[definition]{Proposition}
\newtheorem{lemma}[definition]{Lemma}
\newtheorem{corollary}[definition]{Corollary}

\numberwithin{equation}{section}

\usepackage{geometry}
\renewcommand{\epsilon}{\varepsilon}

\renewcommand{\leq}{\leqslant}
\renewcommand{\le}{\leqslant}
\renewcommand{\geq}{\geqslant}
\renewcommand{\ge}{\geqslant}

\allowdisplaybreaks

 \title[Superposition of fractional Laplacians]{The Neumann condition\\
for the superposition of fractional Laplacians}

\author[S. Dipierro]{Serena Dipierro}

\author[E. Proietti Lippi]{Edoardo Proietti Lippi}

\author[C. Sportelli]{Caterina Sportelli}

\author[E. Valdinoci]{Enrico Valdinoci}

 \address{Department of Mathematics and Statistics
 \newline\indent University of Western Australia \newline\indent
 35 Stirling Highway, WA 6009 Crawley, Australia.\newline
 \newline\indent
\tt serena.dipierro@uwa.edu.au \newline\indent
\tt edoardo.proiettilippi@uwa.edu.au \newline\indent
\tt caterina.sportelli@uwa.edu.au \newline\indent
\tt enrico.valdinoci@uwa.edu.au}
 
\date{}

\begin{document}

 \maketitle

\begin{abstract}
We present a new functional setting for
Neumann conditions related to the
superposition of (possibly infinitely many) fractional Laplace operators.

We will introduce some bespoke
functional framework and present minimization properties, 
existence and uniqueness results, asymptotic formulas, spectral analyses, rigidity results,
integration by parts formulas, superpositions of fractional perimeters, as well as a study of the associated heat equation.
\end{abstract}
 
\tableofcontents{}
 
\section{Introduction} 

\subsection{Superposition of fractional Laplacians and Neumann conditions}

Nonlocal operators have significantly impacted contemporary research, playing a crucial role in both theoretical studies and practical applications across various fields of applied sciences. In this trend, the analysis of models arising from
the superposition of operators of different orders is also receiving increasing attention also
to improve our understanding of complex systems with long-range interactions acting
on different scales,
to allow for a more accurate and realistic representation of several physical phenomena
and to possibly increase computational power in numerical simulations.

In this paper, a novel functional framework for Neumann conditions is introduced, specifically tailored to address the superposition of fractional Laplace operators, potentially involving an infinite (even uncountable) number of such operators. The results obtained include, among the others, minimization properties, existence and uniqueness theorems, asymptotic behavior, spectral analyses, continuity results, integration by parts formulas and divergence theorems, as well as a study of the associated heat equation.

The mathematical setting that we can consider with our methods is  very broad but, for the sake of concreteness, we focus on the following specific framework.

%Let $\sigma_0, \sigma_1\in (0, 1)$ with $\sigma_0<\sigma_1$ and we assume that $s\in [\sigma_0, \sigma_1]$. 
%The assumptions on $\mu$ are the following ones. 

Let $\mu$ be a nonnegative and nontrivial finite (Borel) measure over~$(0, 1)$ and
%Under these assumptions,  there exists $s_\sharp\in (0, 1)$ such that
%\begin{equation}\label{ssharp}
%\mu([s_\sharp, 1))>0.
%\end{equation}
%We will see later that the exponent $s_\sharp$ plays the role of a critical exponent (therefore, roughly
%speaking, while some arbitrariness is allowed in the choice of $s_\sharp$ here above, the results obtained
%will be stronger if one picks $s_\sharp$ ``as large as possible" but still verifying~\eqref{ssharp}).
let~$\alpha \ge 0$. The operator that we address here deals with the possible coexistence of a Laplacian and a possibly continuous superposition of fractional Laplacians, namely
\begin{equation}\label{ELLE}
L_{\alpha, \mu} (u):= -\alpha\Delta u +\int_{(0, 1)} (-\Delta)^s u\, d\mu(s).
\end{equation}
As customary, the notation $(-\Delta)^s$ stands for the fractional Laplacian, defined, for all $s\in (0, 1)$, as 
\begin{equation}\label{deflaplacianofrazionario}
(- \Delta)^s\, u(x) = c_{N,s}\int_{\R^N} \frac{2u(x) - u(x+y)-u(x-y)}{|y|^{N+2s}}\, dy.
\end{equation}
The positive normalizing constant~$c_{N,s}$ is chosen in such a way to provide
consistent limits as~$s\nearrow1$ and as~$s\searrow0$, namely
\[
\lim_{s\nearrow1}(-\Delta)^su=(-\Delta)^1u=-\Delta u
\qquad{\mbox{and}}\qquad
\lim_{s\searrow0}(-\Delta)^s u=(-\Delta)^0u=u.
\]
The explicit value of~$c_{N,s}$ does not play a major role in our paper and is given by
$$ c_{N,s} := -\frac{2^{2s-1}\Gamma(\frac{N+2s}2)}{\pi^{N/2}\Gamma(-s)} $$
see e.g.~\cite[Definition~1.1]{2024arXiv241118238A}. 

In the nonlocal context, at least three notions of fractional normal derivatives have been
considered. One of them is of ``geometric'' type, and deals with the incremental quotient in the normal direction related to the half of the order of the operator: this notion is typically
very useful for boundary regularity, Pohozaev-type identities, Hopf-type results, moving plane methods, symmetry results, etc. (see in particular~\cite{MR3168912, MR3211861, MR3177650, MR3395749, MR3937999}).
This notion gives rise
to overdetermined problems when combined to Dirichlet external data,
but when one allows suitable larger function spaces, the problem with this type of Neumann condition becomes well-posed, see~\cite{MR3293447, MR3920522}.

A second notion of fractional boundary derivative arises from spectral considerations,
in which, roughly speaking, one considers, in lieu of~\eqref{deflaplacianofrazionario}, a different
operator of fractional order obtained by eigenfunction expansion and multiplication by powers of eigenvalues (the Neumann condition being thus encoded into the choice of eigenfunctions, see~\cite{MR2754080, MR3082317, MR3385190}).
The structure of this spectral operator is however quite different from that
of the integral fractional Laplacian in~\eqref{deflaplacianofrazionario}, see in particular~\cite[Section~2.3]{MR3967804}
for similarities and differences.

Another notion of fractional normal derivative has been put forth in~\cite{MR3651008}
and it is nicely compatible with energy methods and variational techniques; also
it has a natural probabilistic interpretation which makes it suitable for applications in mathematical biology (see~\cite{MR4249816, MR4651677}). This latter notion of fractional derivative
related to~$(-\Delta)^s$ takes the form
\begin{equation}\label{defNsu}
\Ns u(x):=c_{N,s}\int_\Omega\frac{u(x)-u(y)}{|x-y|^{N+2s}}\,dy.
\end{equation}

While the first notion of fractional derivative is strictly linked to the order of the fractional operator (or of the induced boundary regularity)
and the second to the spectral structure of the classical Laplacian, the latter setting happens to be conveniently flexible and can deal with the case of operators of mixed order as well. In fact, as we aim to show in this paper, it can be advantageously exploited in a very general framework comprising also
the superposition of possibly infinitely many fractional Laplacians (and possibly a continuum of operators). The setting provided here is actually prone to further generalizations, but we present
the results in their simplest possible form. For this, we define the Neumann conditions associated to the operator $L_{\alpha, \mu}$ in~\eqref{ELLE} in the following way:
\begin{definition}\label{NDEFN}
We say that $u$ satisfies the $(\alpha,\mu)$-\textit{Neumann conditions} if
\begin{equation}\label{alfamucondition}
\begin{cases}
\displaystyle\int_{(0,1)} \Ns u(x)\,d\mu(s)=g(x)\,\, \mbox{ for all }x\in \R^N \setminus \overline{\Omega},    & \mbox{ if }\alpha=0,
\\ \\
\partial_\nu u(x)=h(x) \,\, \mbox{ for all } x\in \partial\Omega,
&  \mbox{ if } \mu\equiv 0,
\\ \\
\partial_\nu u(x)=h(x) \,\, \mbox{ for all } x\in \partial\Omega
 \mbox{ and } \\ 
 \displaystyle\int_{(0,1)} \Ns u(x)\,d\mu(s)=g(x)\,\, \mbox{ for all }x\in \R^N \setminus \overline{\Omega},  
 &\mbox{ if } \alpha\neq 0 \mbox{ and } \mu\not \equiv 0.
\end{cases}
\end{equation}

Moreover, when $g\equiv 0$ in $\R^N\setminus \overline{\Omega}$ 
and $h\equiv 0$ on $\partial \Omega$, we call the conditions in~\eqref{alfamucondition} as \textit{homogeneous $(\alpha,\mu)$-Neumann conditions}.
\end{definition}

It is worth noting that continuous superpositions of operators of different fractional orders have been recently considered in~\cite{MR4736013, DPSV-P, MR4821750}
in the more general case in which~$\mu$ is a signed measure.  We refer the interested reader also to \cite{MR4793906} for the case of combinations of fractional Laplacians with the ``wrong sign".
However,  the setting of~\cite{MR4736013, DPSV-P, MR4821750} is quite different from the one presented here, since it was tailored for problems subject to Dirichlet boundary conditions, while the main focus here is about conditions of Neumann type.

\subsection{Main results}
As a first result, we
show that the functions minimizing the integral of the Gagliardo seminorms
(which will be defined precisely in~\eqref{defsemipi659uir}), namely
\[
\int_{(0, 1)} [u]^2_s\,d\mu(s),
\]
are the ones satisfying the homogeneous Neumann condition
\begin{equation}\label{Neumannomogeneo}
\int_{(0, 1)} \Ns u(x) \,d\mu(s)=0
\end{equation}
for every $x\in \R^N\setminus \overline{\Omega}$ (and this is interesting, because
it relates the homogeneous Neumann condition directly to a variational problem).

\begin{theorem}\label{TH1ojc}
Let $u:\R^N\to \R$ with $u\in L^1(\Omega)$. For all $x\in \R^N\setminus \overline{\Omega}$, we define
\[
E_u(x):=\int_{(0, 1)}
c_{N,s}\int_\Omega \frac{u(z)}{|x-z|^{N+2s}}\,dz\,d\mu(s).
\]
We set
\begin{equation}\label{definizioneuutilde}
\tilde{u}(x):=
\begin{cases}
u(x)   & \mbox{ if } x\in \Omega,
\\ \\
\dfrac{E_u(x)}{E_1(x)}   &\mbox{ if } x\in\R^N\setminus \overline{\Omega},
\end{cases}
\end{equation}
where $E_1$ stands for $E_u$ when $u\equiv 1$. 

Then,
\begin{equation}\label{seminormagagliardo<}
\int_{(0, 1)}c_{N,s} \, [\widetilde u]_s^2 \,d\mu(s)\le \int_{(0, 1)}c_{N,s}\, [u]^2_s \,d\mu(s).
\end{equation}
\end{theorem}

We also give an existence and uniqueness result for solutions of
\begin{equation}\label{prob1p}
\begin{cases}
L_{\alpha, \mu}(u) = f\quad \mbox{ in } \Omega,\\
\mbox{with }(\alpha,\mu)\mbox{-Neumann conditions}.
\end{cases}
\end{equation}
In the forthcoming Section~\ref{KDMCLDFMVV} we will introduce
the precise functional setting for this problem,
see in particular formulas~\eqref{Hmu} and~\eqref{HALPHAMU} for the definition
of the functional spaces~$\mathcal H_\mu (\Omega)$ and~$\mathcal{H}_{\alpha,\mu}(\Omega)$, respectively.

\begin{theorem}\label{uptoconstant}
Let $\Omega$ be a bounded Lipschitz domain. Assume that
$f\in L^2(\Omega)$, $g\in L^1(\R^N\setminus\overline\Omega)$ and $h\in L^1(\partial\Omega)$. Suppose that there exists $\psi \in C^2(\R^N)$ such that $\partial_\nu \psi =h$ on $\partial \Omega$ and
\[
\int_{(0, 1)}\Ns \psi \,d\mu(s)=g \mbox{ in } \R^N\setminus\overline{\Omega}.
\]

Then, problem~\eqref{prob1p} admits a solution in 
$\mathcal{H}_{\alpha,\mu}(\Omega)$ if and only if 
\begin{equation}\label{condizioneesistenzaunicità}
\int_\Omega f\,dx =-\int_{\R^N\setminus\Omega} g\,dx
-\alpha \int_{\partial\Omega}h\,d\mathscr{H}_x^{N-1}.
\end{equation}
Moreover, if~\eqref{condizioneesistenzaunicità} holds,
the solution is unique up to an additive constant.
\end{theorem}

We observe that the statement of Theorem~\ref{uptoconstant} is in analogy with the case~$\mu\equiv 0$ (see~\cite[page~294]{MR3012036}) and with the fractional case in which~$\alpha=0$ and~$\mu$ is a Dirac delta centered at some~$s\in(0,1)$ (see~\cite[Theorem 3.9]{MR3651008}). The proof of Theorem \ref{uptoconstant} follows the same strategy of \cite[Theorem 3.9]{MR3651008} and therefore it is postponed to Appendix \ref{appendix1.3}.

The following result deals with the behavior at infinity of 
functions satisfying the homogeneous Neumann condition
in $\R^N\setminus \overline{\Omega}$.

\begin{theorem}\label{ASYM}
Let $\Omega$ be a bounded open subset of $\R^N$.
Let $u\in \mathcal{H}_\mu(\Omega)$ be a bounded function such that
\[
\int_{(0,1)} \Ns u\,d\mu(s)=0 \,\,
\mbox{ in } \R^N \setminus \overline{\Omega}.
\]
Then
\[
\lim_{|x|\to +\infty}u(x)
=\frac{1}{|\Omega|}\int_\Omega u(x)\,dx.
\]
\end{theorem}

We also study the eigenvalue problem associated with our superposition operator under Neumann conditions:

\begin{theorem}\label{ETGBAVA}
Let $\Omega\subset\R^n$ be a bounded Lipschitz domain. 

Then, there exists a sequence of eigenvalues
\[
0=\lambda_1<\lambda_2\le\lambda_3\dots
\]
with a corresponding sequence of eigenfunctions $u_i:\R^n\to\R$,
with~$i\in\{1,2,\dots\}$, which provide a complete orthogonal system in~$L^2(\Omega)$.
\end{theorem}

Interestingly, the Neumann condition forces a global continuity property, as stated in the following result:

\begin{theorem}\label{propcontinuità}
Let $\Omega\subset \R^N$ be a domain with $C^1$ boundary. Let 
$u$ be continuous in $\overline{\Omega}$, with 
\begin{equation}\label{NNN}
\int_{(0,1)}\Ns u\,d\mu(s)=0 \mbox{ in } \R^N\setminus \overline{\Omega}.
\end{equation}
Then, $u$ is continuous in the whole of $\R^N$.
\end{theorem}

These main results will be complemented by other auxiliary results
which have their own interest, such as
continuous and compact embeddings (Proposition~\ref{propcompactembedding}),
divergence theorems and integration by parts formulas
(Lemmata~\ref{lemmathdivergenza}
and~\ref{lemmaperparti}),
maximum principles (Lemma~\ref{lemmasoluzionecostante}),
Poincaré inequalities (Lemma~\ref{POINCARE}),
mass conservation, energy decresingness and asymptotics for the heat equation
(Propositions~\ref{proposizionemassacostante}, \ref{DECENSD} and~\ref{ASUYY}),
asymptotic properties of a normalized Neumann condition (Proposition~\ref{ASMDNSDIKF}), 
superpositions of fractional perimeters, etc.

\subsection{Examples and possible applications}
The setting that we introduce here for the Neumann conditions
extends the one proposed in~\cite{MR3651008} but is also entirely new in its wide generality. Moreover, the assumptions on the measure~$\mu$ allow us to consider many specific cases which, to the best of our knowledge, have never been addressed in the literature.  We provide some examples\footnote{As customary,
$\delta_s $ denotes the Dirac measure concentrated at some fractional power $s\in (0, 1)$.} below:
\begin{itemize}
\item If $\mu=\delta_s$ for some~$s\in(0,1)$ and~$\alpha=0$, then the operator introduced~\eqref{ELLE} reduces to the fractional Laplacian, for some~$s\in (0, 1)$. In this case, the Neumann conditions in~\eqref{alfamucondition} turn out to be
\[
\Ns u(x) = g(x) \quad\mbox{ for all }x\in \R^N \setminus \overline{\Omega}.
\]
In this case, we find exactly the setting introduced in~\cite{MR3651008};
\item If $\mu=\beta\delta_s$ with $\beta>0$ and $\alpha\neq 0$, then the operator defined in~\eqref{ELLE} reduces to the mixed local and nonlocal problem
\begin{equation}\label{PS23ENt45I2S1}
L_{\alpha, \mu} (u) = -\alpha \Delta u + \beta (-\Delta)^s u
\end{equation}
and the Neumann conditions in~\eqref{alfamucondition} simplify to
\begin{equation}\label{PS23ENt45I2S}
\begin{cases}
\partial_\nu u(x)=h(x) &\mbox{ for all } x\in \partial\Omega,\\
\Ns u(x) = g (x)&\mbox{ for all }x\in \R^N \setminus \overline{\Omega}.
\end{cases}
\end{equation}
The mixed operator~\eqref{PS23ENt45I2S1} has been widely studied in the recent literature
also in regards to 
elliptic regularity~\cite{%MR2095633, MR2180302, 
MR2129093, MR2928344, MR2912450, MR3724879, MR4387204, MR4381148, MR4693935},
parabolic estimates~\cite{MR2653895, MR2911421}, 
%pseudodifferential operators~\cite{MR1763145, MR1874077}, 
classification and symmetry results~\cite{MR3485125},
geometric and functional inequalities~\cite{MR4391102, MR4645045},
numerical schemes~\cite{MR2679574},
Aubry-Mather theory~\cite{MR2542727},
transport in magnetic fields~\cite{BLZ},
%analysis of pandemics~\cite{EPSTEIN}, 
etc.

The study of this mixed operator problem under the Neumann condition~\eqref{PS23ENt45I2S}
has emerged in~\cite{MR4438596, MR4651677} and has been further extended in \cite{MPL22}.

\item Let $n\in\N$ with $n\ge 2$ and consider the measure given by
\begin{equation}\label{ESSOMMAFINITA}
\mu:= \sum_{k=1}^n \delta_{s_k}.
\end{equation}
If $\alpha\neq 0$, the operator in~\eqref{ELLE} provides the superposition of a Laplacian and $n$ nonlocal operators with different orders,  namely
\[
L_{\alpha, \mu} (u) = -\alpha \Delta u + \sum_{k=1}^n (-\Delta)^{s_k} u.
\]
Also, the Neumann conditions in~\eqref{alfamucondition} reduce to
\[
\begin{cases}
\partial_\nu u(x)=h(x) & \mbox{ for all } x\in \partial\Omega,\\
\displaystyle\sum_{k=1}^n \mathscr{N}_{s_k} u (x)= g(x) &\mbox{ for all }x\in \R^N \setminus \overline{\Omega}.
\end{cases}
\]
At the best of our knowledge, this setting is new in the literature;
\item If $\alpha=0$ and $\mu$ is as in~\eqref{ESSOMMAFINITA}, we have the superpositions of $n$ fractional Laplacians, i.e.
\[
L_{\alpha, \mu} (u) = \sum_{k=1}^n (-\Delta)^{s_k} u
\]
with the Neumann conditions in~\eqref{alfamucondition} given by
\[
\sum_{k=1}^n\mathscr N_{s_k} u(x) = g(x) \quad\mbox{ for all }x\in \R^N \setminus \overline{\Omega}.
\]
This case is also new;
\item If $\alpha\neq 0$ and 
\begin{equation}\label{ESSERIE}
\mu:= \sum_{k=1}^{+\infty} c_k \delta_{s_k},
\end{equation}
with $c_k\ge 0$ for any $k\in\N\setminus\{0\}$, then we are able to address the case of a mixed order operator consisting of a Laplacian and a convergent series of infinitely many fractional Laplacians, namely the operator in~\eqref{ELLE} reads as
\[
L_{\alpha, \mu} (u) = -\alpha \Delta u + \sum_{k=1}^{+\infty} c_k\, (-\Delta)^{s_k} u.
\]
In this case, the Neumann conditions reduce to 
\[
\begin{cases}
\partial_\nu u(x)=h(x)& \mbox{ for all } x\in \partial\Omega,\\
\displaystyle\sum_{k=1}^{+\infty} c_k \, \mathscr{N}_{s_k} u(x) = g(x)& \mbox{ for all }x\in \R^N \setminus \overline{\Omega}.
\end{cases}
\]
As far as we know, this problem happens to be new as well;
\item If $\alpha=0$ and $\mu$ is defined as in~\eqref{ESSERIE} with $c_k\ge 0$ for any $k\in\N\setminus\{0\}$, then the operator in~\eqref{ELLE} is written as a convergent series of nonlocal operators, that is
\[
L_{\alpha, \mu} (u) = \sum_{k=1}^{+\infty} c_k\, (-\Delta)^{s_k} u.
\]
Moreover, the Neumann conditions in~\eqref{alfamucondition} are given by
\[
\sum_{k=1}^{+\infty}\mathscr N_{s_k} u (x)= g(x) \quad\mbox{ for all }x\in \R^N \setminus \overline{\Omega}.
\]
Also this case has not been explored in the existing literature;
\item Given a measurable, nonnegative and not identically zero function $f$, we consider the continuous superposition of fractional operators corresponding to the measure $\mu$ such that
\begin{equation}\label{DMU}
d\mu(s):= f(s) \, ds,
\end{equation}
where $ds$ stands for the usual Lebesgue measure. In this case,  if $\alpha\neq 0$, the operator in~\eqref{ELLE} boils down to
\[
L_{\alpha, \mu} (u) = -\alpha \Delta u + \int_0^1 f(s) (-\Delta)^s u \, ds.
\]
With this choice, the Neumann conditions are given by
\[
\begin{cases}
\partial_\nu u(x)=h (x)& \mbox{ for all } x\in \partial\Omega,\\
\displaystyle\int_0^1 f(s)\, \Ns u(x) \, ds = g(x)&\mbox{ for all }x\in \R^N \setminus \overline{\Omega}.
\end{cases}
\]
To the best of our knowledge,  the existing literature lacks of this specific case.
\item If $\alpha=0$ and the measure $\mu$ satisfies~\eqref{DMU}, then we can cover the case of a continuous superposition of fractional operators of the form
\[
L_{\alpha, \mu} (u) =  \int_0^1 f(s) (-\Delta)^s u \, ds
\]
and the Neumann conditions in~\eqref{alfamucondition} reduce to
\[
\int_0^1 f(s)\, \Ns u(x) \, ds = g(x) \quad\mbox{ for all }x\in \R^N \setminus \overline{\Omega}.
\]
\end{itemize}

\subsection{Organization of the paper}
The rest of this paper is organized as follows. The functional framework needed for this paper will be presented in Section~\ref{KDMCLDFMVV}. 

Section~\ref{JOSDNKDJO-1} will put forth
the integration by parts formulas, discuss minimization properties
and give the proof of Theorem~\ref{TH1ojc}.

Section~\ref{PSJLDSLD-2e0dfjev} will make precise the notion of weak solutions and provide the proofs of Theorem~\ref{ASYM}.

The spectral theory and the proof of Theorem~\ref{ETGBAVA} will be contained in Section~\ref{xPSJODLOJDFLFVOLNF023woer}.

The heat equation will be discussed in Section~\ref{KDMCLDFMVV2}
and Section~\ref{KDMCLDFMVV3} will be devoted to the
proof of Theorem~\ref{propcontinuità} (an alternative proof of
Theorem~\ref{propcontinuità} being showcased in Appendix~\ref{KDMCLDFMVV4}).

The connection with fractional perimeters will be presented in Section~\ref{OSJDNDNMwefV}.

Appendix \ref{appendix1.3} contains the proof of Theorem \ref{uptoconstant}. 

Appendices \ref{appendixHilbert}, \ref{appendixfunctional} and \ref{appendixASUYY} contain some results which easily stem from \cite[Proposition 3.1, Proposition 3.7, Proposition 4.3]{MR3651008}. For the sake of completeness, we provide them with all the details.

\section{Functional setting}\label{KDMCLDFMVV}

In this section, we introduce the functional setting that we work in. 
To maintain the discussion as simple as possible, we follow the approach in~\cite{MR3651008};
a different approach with a more delicate choice of test spaces would require the finer analysis put forth in~\cite{MR4683738}.

We start by considering the case in which only nonlocal interactions are present,
that is~$\alpha=0$.
To do this, we introduce the Gagliardo seminorm, that is 
\begin{equation}\label{defsemipi659uir}
[u]_s:= \left({c_{N, s}}\iint_{\mathcal Q} \frac{|u(x)-u(y)|^2}{|x-y|^{N+2s}} \, dx\,dy\right)^{1/2},
\end{equation}
where $\Q:=\R^{2N}\setminus (\R^N\setminus \Omega)^2$.

Given~$g\in L^1(\R^N\setminus\Omega)$, we also define the norm
\begin{equation}\label{normamu}
\|u\|_\mu := \left(\|u\|^2_{L^2(\Omega)} + \| |g|^{1/2} u\|^2_{L^2(\R^N\setminus\Omega)} +\frac12\int_{(0, 1)} [u]_s^2 \, d\mu(s) \right)^{1/2}
\end{equation}
as well as the space $\mathcal H_\mu (\Omega)$ of all measurable functions such that the norm in~\eqref{normamu} is finite, namely
\begin{equation}\label{Hmu}
\mathcal H_\mu (\Omega):= \left\lbrace u:\R^N\to\R \text{ measurable } : \|u\|_\mu <+\infty \right\rbrace.
\end{equation}

To deal with the more general case, 
we consider functions~$g\in L^1(\R^N\setminus\Omega)$ and~$h\in L^1(\partial\Omega)$ and define the norm
\begin{equation}\label{normaalfamu}
\begin{split}
\|u\|_{\alpha, \mu} := &\Bigg(\|u\|^2_{L^2(\Omega)} + \||h|^{1/2} u\|^2_{L^2(\partial\Omega)} + \| |g|^{1/2} u\|^2_{L^2(\R^N\setminus\Omega)} \\
&\qquad\qquad+ \alpha \|\nabla u\|^2_{L^2(\Omega)} +\frac12 \int_{(0, 1)} [u]_s^2 \, d\mu(s) \Bigg)^{1/2}
\end{split}
\end{equation}
and the space
\begin{equation}\label{HALPHAMU}
\mathcal H_{\alpha, \mu} (\Omega):= \begin{cases}
H^1(\Omega) &\hbox{ if } \mu \equiv 0,\\
\mathcal H_\mu (\Omega) &\hbox{ if } \alpha = 0,\\
H^1(\Omega)\cap \mathcal H_\mu (\Omega) &\hbox{ if }  \mu \not\equiv 0 \text{ and } \alpha\neq 0,
\end{cases}
\end{equation}
endowed with the norm in~\eqref{normaalfamu}.

It is easy to check that the norms defined in~\eqref{normamu} and~\eqref{normaalfamu} are induced by the scalar products
\begin{equation}\label{prodottoscalaremu}
(u,v)_\mu:=\int_\Omega uv\,dx
+\int_{\R^N\setminus \Omega}|g|uv\,dx
+\int_{(0,1)} \frac{c_{N,s}}{2}
\iint_\Q \frac{(u(x)-u(y))(v(x)-v(y))}{|x-y|^{N+2s}}\,dx\,dy\,d\mu(s)
\end{equation}
and
\begin{equation}\label{prodottoscalarealfamu}
\begin{aligned}
(u,v)_{\alpha,\mu}&:=\int_\Omega uv\,dx
+\int_{\partial \Omega}|h|uv\,d\mathscr H_x^{N-1}
+\int_{\R^N\setminus \Omega}|g|uv\,dx
+\alpha\int_\Omega \nabla u\cdot \nabla v\,dx \\
&\qquad+\int_{(0,1)} \frac{c_{N,s}}{2}
\iint_\Q \frac{(u(x)-u(y))(v(x)-v(y))}{|x-y|^{N+2s}}\,dx\,dy\,d\mu(s),
\end{aligned}
\end{equation}
respectively.

\begin{proposition}
$(\mathcal{H}_{\alpha,\mu}(\Omega), \|\cdot\|_{\alpha, \mu})$ is a Hilbert space.
\end{proposition}

\begin{proof}
We observe that
\begin{equation}\label{Hmuhilbert}
(\mathcal{H}_\mu(\Omega), \|\cdot\|_{\mu}) \mbox{ is a Hilbert space.}
\end{equation}
(see Appendix \ref{appendixHilbert}).
Hence, the completeness of $\mathcal{H}_{\alpha,\mu}(\Omega)$ follows from its
definition in~\eqref{HALPHAMU} and~\ref{Hmuhilbert}.

Accordingly, observing that~\eqref{prodottoscalarealfamu} is a bilinear form and~$\|u\|_{\alpha,\mu}=(u,u)_{\alpha,\mu}^{1/2}$, it only remains to show that if~$\|u\|_{\alpha,\mu}=0$, then~$u=0$ a.e. in~$\R^N$ (or, if~$\mu\equiv 0$, in~$\Omega$).

To this aim, we observe that, if $\|u\|_{\alpha,\mu}=0$, then~$\|u\|_{L^2(\Omega)}=0$, and thus~$u=0$ a.e. in~$\Omega$. Hence, if~$\mu\equiv 0$, the proof is concluded. Otherwise, we also have
\[
\int_{(0,1)} \left(c_{N,s}\iint_\Q \frac{|u(x)-u(y)|^2}{|x-y|^{N+2s}}\,dx\,dy\right)\,d\mu(s) =0,
\]
which gives
\[
\iint_\Q \frac{|u(x)-u(y)|^2}{|x-y|^{N+2s}}\,dx\,dy =0 \quad\mbox{ for any } s\in\mbox{supp}(\mu).
\]
Hence $|u(x)-u(y)|=0$ for any $(x, y)\in\Q$. In particular, for a.e. $x\in \R^n\setminus\Omega$ and $y\in\Omega$, we get
\[
u(x) = u(x) -u(y)=0,
\]
which means that $u=0$ for a.e. $x\in\R^N$. This concludes the proof.
\end{proof}

As a next step, we aim to prove an embedding result for the spaces $\mathcal{H}_\mu(\Omega)$ and $\mathcal{H}_{\alpha,\mu}(\Omega)$. 

We notice that, since $\mu$ is nontrivial,  there exists $s_\sharp\in (0, 1)$ such that
\begin{equation}\label{ssharp}
\mu([s_\sharp, 1))>0.
\end{equation}
We will show that the space $\mathcal H_\mu(\Omega)$ is continuously embedded in $L^{2^*_{s_\sharp}}(\Omega)$, with\footnote{We remark that some arbitrariness is allowed in the choice of $s_\sharp$ in~\eqref{ASHP}, hence the results obtained will be stronger if one picks $s_\sharp$ ``as large as possible" but still verifying~\eqref{ssharp}.}
\begin{equation}\label{ASHP}2^*_{s_\sharp}:=\frac{2N}{N-2s_\sharp}.\end{equation}

For this, we provide a preliminary result.

\begin{lemma}\label{lemmaGG2.1}
Let $\underline s\in (0, 1)$. Let also $\underline s\le s_1\le s_2 < 1$ and $\Omega$ be an open subset of $\R^N$. Let~$u:\Omega\to \R$ be a measurable function. Then,
\[
\|u\|_{L^2(\Omega)}^2 +\iint_{\Omega\times \Omega}\frac{|u(x)-u(y)|^2}{|x-y|^{N+2s_1}}\,dx\,dy\le c\left(\|u\|_{L^2(\Omega)}^2 +\iint_{\Omega\times \Omega}\frac{|u(x)-u(y)|^2}{|x-y|^{N+2s_2}}\,dx\,dy\right),
\]
where $c=c(N, \underline s)\ge 1$.

Explicitly, one can take
\[
c(N, \underline s):= 1+\frac{2 \,\omega_{N-1}}{\underline s},
\]
where $\omega_{N-1}$ denotes the measure of the unit ball in $\R^{N}$.
\end{lemma}

\begin{proof}
We write
\[
\begin{split}
&\iint_{\Omega\times \Omega}\frac{|u(x)-u(y)|^2}{|x-y|^{N+2s_1}}\,dx\,dy\\
&\qquad = \iint_{(\Omega\times\Omega)\cap\{|x-y|< 1\}}\frac{|u(x)-u(y)|^2}{|x-y|^{N+2s_1}}\,dx\,dy + \iint_{(\Omega\times\Omega)\cap\{|x-y|\geq 1\}}\frac{|u(x)-u(y)|^2}{|x-y|^{N+2s_1}}\,dx\,dy.
\end{split}
\]
Notice that, since $s_1 \le s_2$, we have
\begin{equation}\label{BOH1}
\begin{split}
\iint_{(\Omega\times\Omega)\cap\{|x-y|< 1\}}\frac{|u(x)-u(y)|^2}{|x-y|^{N+2s_1}}\,dx\,dy &\le \iint_{(\Omega\times\Omega)\cap\{|x-y|< 1\}}\frac{|u(x)-u(y)|^2}{|x-y|^{N+2s_2}}\,dx\,dy\\
&\le \iint_{\Omega\times\Omega}\frac{|u(x)-u(y)|^2}{|x-y|^{N+2s_2}}\,dx\,dy.
\end{split}
\end{equation}
Moreover, since $N+2s_1>N$ and $s_1\ge \underline{s}>0$, we have
\begin{equation}\label{BOH2}
\begin{split}
&\iint_{(\Omega\times\Omega)\cap\{|x-y|\ge 1\}}\frac{|u(x)-u(y)|^2}{|x-y|^{N+2s_1}}\,dx\,dy\\
& \le 2\left(\int_{\Omega}|u(x)|^2 dx \int_{|z|\ge 1}\frac{dz}{|z|^{N+2s_1}} + \int_{\Omega}|u(y)|^2 dy \int_{|z|\ge 1}\frac{dz}{|z|^{N+2s_1}}\right)\le \frac{2\, \omega_{N-1}}{\underline s}\|u\|^2_{L^2(\Omega)}.
\end{split}
\end{equation}
Thus,  combining~\eqref{BOH1} and~\eqref{BOH2}, we have
\[
\begin{split}
\|u\|_{L^2(\Omega)}^2 +\iint_{\Omega\times \Omega}\frac{|u(x)-u(y)|^2}{|x-y|^{N+2s_1}}\,dx\,dy &\le \left(1+\frac{2\, \omega_{N-1}}{\underline s} \right)\|u\|^2_{L^2(\Omega)} + \iint_{\Omega\times\Omega}\frac{|u(x)-u(y)|^2}{|x-y|^{N+2s_2}}\,dx\,dy\\
&\le c\left(\|u\|_{L^2(\Omega)}^2 +\iint_{\Omega\times \Omega}\frac{|u(x)-u(y)|^2}{|x-y|^{N+2s_2}}\,dx\,dy\right),
\end{split}
\]
as desired.
\end{proof}

Thanks to Lemma~\ref{lemmaGG2.1}, we are able to prove an embedding result for the space~$\mathcal H_\mu(\Omega)$ defined in~\eqref{Hmu}.

\begin{proposition}\label{propcompactembedding}
Let~$\Omega$ be a bounded, open set in~$\R^N$ with Lipschitz boundary. 
Let $s_\sharp$ be as in~\eqref{ssharp}. Then, the space $\mathcal{H}_\mu(\Omega)$ is continuously embedded in~$H^{s_\sharp}(\Omega)$.

In particular, $\mathcal{H}_\mu(\Omega)$ is compactly embedded in~$L^q(\Omega)$ for any~$q\in [1,2^*_{s_\sharp})$.
\end{proposition}

\begin{proof}
We observe that 
\begin{equation}\label{immersione1}
\|u\|_{L^2(\Omega)}^2 +\int_{(0,1)} \frac{c_{N, s}}{2}\iint_{\Omega\times \Omega}
\frac{|u(x)-u(y)|^2}{|x-y|^{N+2s}}\,dx\,dy\,d\mu(s)
\leq \|u\|_\mu^2,
\end{equation}
thanks to~\eqref{normamu}.
Then, we use Lemma~\ref{lemmaGG2.1} with~$s_1:=s_\sharp$ and~$s_2:=s\in [s_\sharp,1)$. By~\eqref{ssharp}, without loss of generality we can assume that there exists~$\delta>0$ such that~$\mu([s_{\sharp}, 1-\delta])>0$.  
We set
\[
m:= \min_{s\in [s_\sharp, 1-\delta]}\left\{\frac{1}{\mu ([s_\sharp, 1-\delta])}, \frac{c_{N, s}}{2} \right\}.
\]
{F}rom this, we have
\begin{equation}\label{immersione2}
\begin{aligned}
\|u\|_{L^2(\Omega)}^2 &+\int_{(0,1)}\frac{c_{N, s}}{2}\iint_{\Omega\times \Omega}\frac{|u(x)-u(y)|^2}{|x-y|^{N+2s}}\,dx\,dy\,d\mu(s)\\
&\ge \|u\|_{L^2(\Omega)}^2 +\int_{[s_\sharp,1-\delta]}\frac{c_{N, s}}{2}\iint_{\Omega\times \Omega}
\frac{|u(x)-u(y)|^2}{|x-y|^{N+2s}}\,dx\,dy\,d\mu(s)\\
&=\int_{[s_\sharp,1-\delta]} \left(\frac{\|u\|_{L^2(\Omega)}^2}{\mu ([s_\sharp, 1-\delta])} + \frac{c_{N, s}}{2} \iint_{\Omega\times \Omega} \frac{|u(x)-u(y)|^2}{|x-y|^{N+2s}}\,dx\,dy \right) d\mu(s)\\
&\ge m \int_{[s_\sharp, 1-\delta])} \left( \|u\|_{L^2(\Omega)}^2 + \iint_{\Omega\times \Omega}\frac{|u(x)-u(y)|^2}{|x-y|^{N+2s}}\,dx\,dy\right) d\mu(s)\\
&\ge \frac{m}{C(N, s_\sharp)} \int_{[s_\sharp,1-\delta]} \left( \|u\|_{L^2(\Omega)}^2 + \iint_{\Omega\times \Omega}\frac{|u(x)-u(y)|^2}{|x-y|^{N+2s_\sharp}}\,dx\,dy\right) d\mu(s)\\
&\geq C_1 \|u\|_{H^{s_\sharp}(\Omega)}^2,
\end{aligned}
\end{equation}
for some $C_1 = C_1 (N, s_\sharp, \delta)>0$. 

Now, combining inequalities~\eqref{immersione1} and~\eqref{immersione2} we arrive at
\[
\|u\|_{H^{s_\sharp}(\Omega)}^2\leq C \|u\|_\mu^2,
\]
up to renaming the constant $C$.

This proves that~$\mathcal{H}_\mu(\Omega)$ is continuously embedded in~$H^{s_\sharp}(\Omega)$, as claimed.  Finally, the compact embedding follows from~\cite[Corollary 7.2]{MR2944369} and this concludes the proof.
\end{proof}

As a consequence of Proposition~\ref{propcompactembedding}, we derive a suitable compact embedding result for the space $\mathcal H_{\alpha, \mu}(\Omega)$ as well.

\begin{corollary}\label{IMMERSIONEALPHAMU}
Let~$\Omega$ be a bounded, open set in~$\R^N$ with Lipschitz boundary.
Assume that $\alpha\neq 0$. Then, the space~$\mathcal H_{\alpha, \mu}(\Omega)$ is continuously embedded in~$H^1(\Omega)$. 

In particular, we have that~$\mathcal H_{\alpha, \mu}(\Omega)$ is compactly embedded in~$L^q(\Omega)$ for any~$q\in [1,2^*)$, where
\[
2^*:=\frac{2N}{N-2}.
\]
\end{corollary}

\begin{proof}
Let $\alpha\neq 0$. We observe that, by the definition of the norm
in~\eqref{normaalfamu}, if~$u\in\mathcal H_{\alpha, \mu}(\Omega)$, then
\[
\|u\|_{H^1(\Omega)} \le C \|u\|_{\alpha, \mu},
\]
for a suitable constant $C>0$.
\end{proof}

\section{Integration by parts formulas and minimization property}\label{JOSDNKDJO-1}

In this section, we will prove some auxiliary formulas which will be useful in the forthcoming results. 

To start, we give the analogue of the divergence theorem. We point out that the assumptions required in the next two results are significantly more general than the ones required in \cite[Lemma 3.2 and Lemma 3.3]{MR3651008}.

\begin{lemma}\label{lemmathdivergenza}
Let $u:\R^N\to \R$ 
be such that $u\in C^2(\Omega)$ and 
\begin{equation}\label{ipotesidivergenza1}
(-\Delta)^su \in L^1(\Omega \times (0,1)).
\end{equation}
Assume that the function
\begin{equation}\label{ipotesidivergenza2}
\Omega\times(\R^N\setminus \Omega)\times (0,1)\ni (x,y,s)
\mapsto c_{N,s}\frac{u(x)-u(y)}{|x-y|^{N+2s}}
\mbox{ belongs to }
L^1(\Omega \times(\R^N\setminus \Omega) \times (0,1)).
\end{equation}

Then,
\begin{equation}\label{tesidivergenza}
\int_{(0, 1)}\int_\Omega (-\Delta)^su(x)\,dx\,d\mu(s)
=\int_{(0, 1)}\int_{\R^N\setminus \Omega}\Ns u(x)\,dx\,d\mu(s).
\end{equation}
\end{lemma}

\begin{proof}
We first point out that the assumptions~\eqref{ipotesidivergenza1}
and~\eqref{ipotesidivergenza2} give that the integrals in 
\eqref{tesidivergenza} are finite.

Now, we observe that
\[
\int_\Omega\int_\Omega \frac{u(x)-u(y)}{|x-y|^{N+2s}}\,dx\,dy
=\int_\Omega\int_\Omega \frac{u(y)-u(x)}{|x-y|^{N+2s}}\,dx\,dy=0,
\]
since the role of $x$ and $y$ is symmetric. 

This implies that
\[
\int_{(0, 1)}c_{N,s}\int_\Omega\int_\Omega \frac{u(x)-u(y)}{|x-y|^{N+2s}}\,dx\,dy \,d\mu(s)=0.
\]
Hence, recalling the definitions in~\eqref{deflaplacianofrazionario} and
\eqref{defNsu}, we can use~\eqref{ipotesidivergenza2} to 
exchange the order of integration and obtain
\[
\begin{aligned}
\int_{(0, 1)}\int_\Omega (-\Delta)^su(x)\,dx\,d\mu(s)&=\int_{(0, 1)}c_{N,s}\int_\Omega \int_{\R^N}
\frac{u(x)-u(y)}{|x-y|^{N+2s}}\,dydx \,d\mu(s) \\
&=\int_{(0, 1)}c_{N,s}\int_\Omega \int_{\R^N\setminus \Omega}\frac{u(x)-u(y)}{|x-y|^{N+2s}}\,dydx \,d\mu(s) \\
&=\int_{(0, 1)}c_{N,s} \int_{\R^N\setminus \Omega}\int_\Omega\frac{u(x)-u(y)}{|x-y|^{N+2s}}\,dx\,dy \,d\mu(s) \\
&=-\int_{(0, 1)}\int_{\R^N\setminus \Omega}\Ns u(y)\,dy\,d\mu(s),
\end{aligned}
\]
as desired.
\end{proof}

We have the following integration by parts formula.

\begin{lemma}\label{lemmaperparti}
Let $u,v:\R^N\to \R$ be such that
$u,v\in C^2(\Omega)$ and
\begin{equation}\label{ipotesiperparti1}
\Q\times(0,1)\ni (x,y,s) \mapsto 
c_{N,s}\frac{(u(x)-u(y))(v(x)-v(y))}{|x-y|^{N+2s}}
\mbox{ belongs to } L^1(\Q\times(0,1)).
\end{equation}
Assume that
\begin{equation}\label{ipotesiperparti2}
(-\Delta)^su v\in L^1(\Omega\times (0,1))
\end{equation}
and that the function
\begin{equation}\label{ipotesiperparti3}\begin{split}
&(\R^N\setminus \Omega)\times \Omega \times (0,1)\ni (x,y,s)
\mapsto c_{N,s}\frac{(u(x)-u(y))v(x)}{|x-y|^{N+2s}}\\&
\mbox{belongs to }
L^1((\R^N\setminus \Omega)\times \Omega \times(0,1)).\end{split}
\end{equation}

Then,
\begin{equation}\label{tesiperparti}
\begin{aligned}
\frac{1}{2}\int_{(0, 1)}c_{N,s}&\iint_\Q\frac{(u(x)-u(y))(v(x)-v(y))}{|x-y|^{N+2s}}\,dx\,dy\,d\mu(s) \\
&=\int_{(0, 1)} \int_\Omega v(x)(-\Delta)^su(x)\,dx\,d\mu(s)+\int_{(0, 1)}\int_{\R^N\setminus \Omega}v(x)\Ns u(x)\,dx \,d\mu(s)
\end{aligned}
\end{equation}
where $c_{N,s}$ is the constant in~\eqref{deflaplacianofrazionario}.
\end{lemma}

\begin{proof}
First, we observe that the assumptions~\eqref{ipotesiperparti1},
\eqref{ipotesiperparti2} and~\eqref{ipotesiperparti3} guarantee 
that the integrals in~\eqref{tesiperparti} are finite.

We notice that 
$\Q=(\Omega\times \Omega)\cup (\Omega \times \R^N\setminus \Omega)\cup ( \R^N\setminus \Omega \times \Omega )$.
Then, we compute
\[
\begin{aligned}
\int_\Omega \int_\Omega &\frac{(u(x)-u(y))(v(x)-v(y))}{|x-y|^{N+2s}}\,dx\,dy \\
&=\int_\Omega \int_\Omega v(x)\frac{u(x)-u(y)}{|x-y|^{N+2s}}\,dx\,dy
-\int_\Omega \int_\Omega v(y)\frac{u(x)-u(y)}{|x-y|^{N+2s}}\,dx\,dy \\
&=2\int_\Omega \int_\Omega v(x)\frac{u(x)-u(y)}{|x-y|^{N+2s}}\,dx\,dy.
\end{aligned}
\]
Moreover,
\[
\begin{aligned}
\int_\Omega \int_{\R^N\setminus\Omega} &\frac{(u(x)-u(y))(v(x)-v(y))}{|x-y|^{N+2s}}\,dx\,dy \\
&=\int_\Omega \int_{\R^N\setminus\Omega} v(x)\frac{u(x)-u(y)}{|x-y|^{N+2s}}\,dx\,dy
-\int_\Omega \int_{\R^N\setminus\Omega} v(y)\frac{u(x)-u(y)}{|x-y|^{N+2s}}\,dx\,dy \\
&=\int_\Omega \int_{\R^N\setminus\Omega} v(x)\frac{u(x)-u(y)}{|x-y|^{N+2s}}\,dx\,dy
+\int_{\R^N\setminus\Omega}\int_\Omega  v(x)\frac{u(x)-u(y)}{|x-y|^{N+2s}}\,dx\,dy,
\end{aligned}
\]
and in a similar way
\[
\begin{aligned}
\int_{\R^N\setminus\Omega} \int_\Omega  &\frac{(u(x)-u(y))(v(x)-v(y))}{|x-y|^{N+2s}}\,dx\,dy \\
&=\int_{\R^N\setminus\Omega} \int_\Omega v(x)\frac{u(x)-u(y)}{|x-y|^{N+2s}}\,dx\,dy
-\int_{\R^N\setminus\Omega} \int_\Omega v(y)\frac{u(x)-u(y)}{|x-y|^{N+2s}}\,dx\,dy \\
&=\int_{\R^N\setminus\Omega} \int_\Omega v(x)\frac{u(x)-u(y)}{|x-y|^{N+2s}}\,dx\,dy
+\int_\Omega \int_{\R^N\setminus\Omega}  v(x)\frac{u(x)-u(y)}{|x-y|^{N+2s}}\,dx\,dy.
\end{aligned}
\]
Combining all the previous identities, we get
\[
\begin{aligned}
\frac{1}{2}\iint_\Q&\frac{(u(x)-u(y))(v(x)-v(y))}{|x-y|^{N+2s}}\,dx\,dy \\
&=\int_\Omega\int_{\R^N}v(x)\frac{u(x)-u(y)}{|x-y|^{N+2s}}\,dx\,dy
+\int_{\R^N\setminus\Omega} \int_\Omega v(x)\frac{u(x)-u(y)}{|x-y|^{N+2s}}\,dx\,dy.
\end{aligned}
\]
Thus, by using~\eqref{deflaplacianofrazionario} and~\eqref{defNsu},
\[
\begin{aligned}
\int_{(0, 1)} \frac{c_{N,s}}{2}&\iint_\Q
\frac{(u(x)-u(y))(v(x)-v(y))}{|x-y|^{N+2s}}\,dx\,dy\,d\mu(s) \\
&=\int_{(0, 1)} \int_\Omega v(x) c_{N,s}
\int_{\R^N}\frac{u(x)-u(y)}{|x-y|^{N+2s}}\,dy \,dx\,d\mu(s) \\
&+\int_{(0, 1)}\int_{\R^N\setminus \Omega}v(x) c_{N,s}
\int_{\Omega}\frac{u(x)-u(y)}{|x-y|^{N+2s}}\,dy \,dx\,d\mu(s)
\\
&=\int_{(0, 1)} \int_\Omega v(x)(-\Delta)^su(x)\,dx\,d\mu(s)
+\int_{(0, 1)}\int_{\R^N\setminus \Omega}v(x)\Ns u(x)\,dx \,d\mu(s).
\end{aligned}
\]
This concludes the proof.
\end{proof}

Now we complete the proof of Theorem~\ref{TH1ojc}.

\begin{proof}[Proof of Theorem~\ref{TH1ojc}]
Without loss of generality, we can suppose that
\[
\int_{(0, 1)}c_{N,s}\iint_\Q \frac{|u(x)-u(y)|^2}{|x-y|^{N+2s}}\,dx\,dy\,d\mu(s)<+\infty,
\]
otherwise we are done.

Now, we observe that 
\begin{equation}\label{seminormaomega^2}
\int_{(0, 1)}c_{N,s}\int_\Omega \int_\Omega
\frac{|\widetilde{u}(x)-\widetilde{u}(y)|^2}{|x-y|^{N+2s}}\,dx\,dy\,d\mu(s)
=\int_{(0, 1)}c_{N,s}\int_\Omega \int_\Omega
\frac{|u(x)-u(y)|^2}{|x-y|^{N+2s}}\,dx\,dy\,d\mu(s),
\end{equation}
so we only need to consider the integral over 
$(\R^N\setminus \Omega)\times \Omega$, or equivalently over 
$\Omega \times (\R^N\setminus \Omega)$.
Setting $\varphi(x):=u(x)-\widetilde{u}(x)$, for every 
$y\in \R^N\setminus \overline{\Omega}$ we have
\begin{equation}\label{seminormavarphi}
\begin{aligned}
\int_{(0, 1)}c_{N,s}\int_\Omega &
\frac{|u(x)-u(y)|^2}{|x-y|^{N+2s}}\,dx\,d\mu(s)
=\int_{(0, 1)}c_{N,s}\int_\Omega
\frac{|u(x)-\widetilde{u}(y)-\varphi(y)|^2}{|x-y|^{N+2s}}\,dx\,d\mu(s)\\
&=\int_{(0, 1)}c_{N,s}\int_\Omega
\frac{|u(x)-\widetilde{u}(y)|^2-2\varphi(y)(u(x)-\widetilde{u}(y))+|\varphi(y)|^2}{|x-y|^{N+2s}}\,dx\,d\mu(s).
\end{aligned}
\end{equation}
We notice that, for every $y\in \R^N\setminus \overline{\Omega}$,
\[
\int_{(0, 1)}c_{N,s}\int_\Omega
\frac{u(x)-\widetilde{u}(y)}{|x-y|^{N+2s}}\,dx\,d\mu(s)
=E_u(y)-\frac{E_u(y)}{E_1(y)}E_1(y)=0.
\]
Therefore we deduce from the identity in~\eqref{seminormavarphi} that
\[
\begin{aligned}
\int_{(0, 1)}c_{N,s}\int_\Omega 
\frac{|u(x)-u(y)|^2}{|x-y|^{N+2s}}\,dx\,d\mu(s)
&=\int_{(0, 1)}c_{N,s}\int_\Omega
\frac{|\widetilde{u}(x)-\widetilde{u}(y)|^2+|\varphi(y)|^2}{|x-y|^{N+2s}}\,dx\,d\mu(s) \\
&\geq \int_{(0, 1)}c_{N,s}\int_\Omega
\frac{|\widetilde{u}(x)-\widetilde{u}(y)|^2}{|x-y|^{N+2s}}\,dx\,d\mu(s)
\end{aligned}
\]
for every $y\in \R^N\setminus \overline{\Omega}$, and the 
equality holds if and only if $\varphi(y)=0$.

Integrating over $\R^N\setminus \Omega$ (or, equivalently, over
$\R^N\setminus \overline{\Omega}$), and using Fubini's theorem,
we conclude that
\[
\int_{(0, 1)}c_{N,s}\int_{\R^N\setminus \Omega}\int_\Omega 
\frac{|u(x)-u(y)|^2}{|x-y|^{N+2s}}\,dx\,dy\,d\mu(s)
\geq \int_{(0, 1)}c_{N,s}\int_{\R^N\setminus \Omega}\int_\Omega
\frac{|\widetilde{u}(x)-\widetilde{u}(y)|^2}{|x-y|^{N+2s}}\,dx\,dy\,d\mu(s),
\]
and the equality holds if and only if $\varphi\equiv 0$ in
$\R^N\setminus \Omega$. This and~\eqref{seminormaomega^2} imply
\eqref{seminormagagliardo<}, as desired.
\end{proof}

\section{Weak solutions with ($\alpha,\mu$)-Neumann conditions}\label{PSJLDSLD-2e0dfjev}

In this section we consider the problem
\begin{equation}\label{prob1}
\begin{cases}
L_{\alpha, \mu}(u) = f \quad\mbox{ in } \Omega,\\
\mbox{with }(\alpha,\mu)\mbox{-Neumann conditions}
\end{cases}
\end{equation}
for some suitable functions $f:\Omega\to\R$, $g:\R^N\setminus\overline{\Omega}\to\R$ and $h:\partial\Omega\to\R$.

We now provide the following definition.
\begin{definition}\label{weakdefn}
Assume that $f\in L^2(\Omega)$, $g\in L^1(\R^N\setminus\overline\Omega)$
and $h\in L^1(\partial\Omega)$. We say that~$u\in \mathcal{H}_{\alpha, \mu}(\Omega)$ is a weak solution of the problem~\eqref{prob1}
%\begin{equation}\label{prob1}
%\begin{cases}
%L_{\alpha, \mu}(u) = f &\mbox{ in } \Omega,\\
%\mbox{with }(\alpha,\mu)\mbox{-Neumann conditions}
%\end{cases}
%\end{equation}
if, for any~$v\in  \mathcal{H}_{\alpha, \mu}(\Omega)$,
\begin{equation}\label{weaksolution}
\begin{split}
&\alpha \int_\Omega \nabla u\cdot\nabla v\, dx +\int_{(0, 1)} \frac{c_{N, s}}{2} \iint_{\Q} \frac{(u(x)-u(y))(v(x)-v(y))}{|x-y|^{N+2s}} \,dx\,dy\, d\mu(s)\\
&\quad = \int_\Omega fv \, dx +\int_{\R^N\setminus\Omega} gv \, dx +\int_{\partial\Omega} hv \, d\mathscr{H}_x^{N-1}.
\end{split}
\end{equation}
\end{definition}

We state the following preliminary result, with the proof provided in Appendix \ref{appendixfunctional}.
\begin{proposition}\label{propfunctional}
Assume that $f\in L^2(\Omega)$, $g\in L^1(\R^N\setminus\Omega)$
and $h\in L^1(\partial\Omega)$. Let $I:\mathcal{H}_{\alpha, \mu}(\Omega)\to\R$ be the functional defined as
\[
\begin{split}
I(u):= \frac{\alpha}{2}\int_\Omega |\nabla u|^2 dx &+\int_{(0, 1)} \frac{c_{N, s}}{4} \iint_{\Q} \frac{|u(x)-u(y)|^2}{|x-y|^{N+2s}} dx\,dy\, d\mu(s)\\
&- \int_\Omega fu \, dx -\int_{\R^N\setminus\Omega} gu \, dx -\int_{\partial\Omega} hu \, d\mathscr{H}_x^{N-1}.
\end{split}
\]
Then, critical points of~$I$ are weak solutions of the problem~\eqref{prob1}.
\end{proposition}

Next result is a sort of maximum principle and is helpful to prove the existence and uniqueness result in Theorem~\ref{uptoconstant}.

\begin{lemma}\label{lemmasoluzionecostante}
Let $f\in L^2(\Omega)$, $g\in L^1(\R^N\setminus\overline\Omega)$ and~$h\in L^1(\partial\Omega)$ be nonnegative functions. Let~$u\in \mathcal{H}_{\alpha,\mu}(\Omega)$ be a weak solution of~\eqref{prob1}.

Then, $u$ is constant.
\end{lemma}

\begin{proof}
The argument is inspired by that of Lemma~3.8 in~\cite{MR3651008}, but some care is needed to address the presence of the measure~$\mu$.
To this end, observing that the constant function $v\equiv 1$ belongs 
to $\mathcal{H}_{\alpha,\mu}(\Omega)$, we can use it as a test function in 
\eqref{weaksolution} to obtain that
\[
0=\int_\Omega f\,dx +\int_{\R^N\setminus\Omega} g\,dx
+\int_{\partial\Omega}h\,d\mathscr{H}_x^{N-1}.
\]
This implies that $f=0$ a.e. in $\Omega$, $g=0$ a.e. in 
$\R^N\setminus\Omega$ and $h=0$ a.e. in $\partial\Omega$.

Thus, taking $v:=u$ in~\eqref{weaksolution}, we get
\[
\alpha \int_\Omega |\nabla u|^2\, dx +\int_{(0, 1)} \frac{c_{N, s}}{2} \iint_{\Q} \frac{|u(x)-u(y)|^2}{|x-y|^{N+2s}} dx\,dy\, d\mu(s)=0,
\]
hence $u$ is constant in $\R^N$, as desired.
\end{proof}

Now we complete the proof of Theorem~\ref{ASYM}
(though inspired by Proposition~3.13 
in~\cite{MR3651008}, the complication arising from the measure does require here some bespoke modifications).

\begin{proof}[Proof of Theorem~\ref{ASYM}]
We observe that, since $\Omega$ is bounded, there exists~$R>0$
such that~$\Omega \subset B_R$.
Thus, if $y\in \Omega$ and $|x|>R$, we have
\[
|x|-R\leq |x-y|\leq |x|+R.
\]
Since $u$ is bounded, we set $\tilde{u}(x):=u(x)+c$ for some
$c>0$ such that $\tilde{u}\geq 0$. We notice that, for any
$x\in \R^N \setminus \overline{\Omega}$,
\[
\int_{(0,1)} \Ns \tilde{u}(x)\,d\mu(s)
=\int_{(0,1)} \Ns u(x)\,d\mu(s)=0.
\]
Thus, recalling the definition in~\eqref{defNsu}, for any 
$x\in \R^N\setminus \overline{\Omega}$ we can write 
\[
\tilde{u}(x)=\dfrac{\displaystyle\int_{(0,1)}c_{N,s}\int_\Omega \frac{\tilde{u}(y)}{|x-y|^{N+2s}}\,dy\,d\mu(s)}{\displaystyle\int_{(0,1)}c_{N,s}\int_\Omega \frac{1}{|x-y|^{N+2s}}\,dy\,d\mu(s)}.
\]
Therefore, if $|x|>R$, from the previous inequality we obtain that
\[
\dfrac{\displaystyle\int_{(0,1)}c_{N,s}\int_\Omega \frac{\tilde{u}(y)}{(|x|+R)^{N+2s}}\,dy\,d\mu(s)}{\displaystyle\int_{(0,1)}c_{N,s}\int_\Omega \frac{1}{(|x|-R)^{N+2s}}\,dy\,d\mu(s)}
\leq \tilde{u}(x) \leq
\dfrac{\displaystyle\int_{(0,1)}c_{N,s}\int_\Omega \frac{\tilde{u}(y)}{(|x|-R)^{N+2s}}\,dy\,d\mu(s)}{\displaystyle\int_{(0,1)}c_{N,s}\int_\Omega \frac{1}{(|x|+R)^{N+2s}}\,dy\,d\mu(s)},
\]
which becomes
\begin{equation}\label{troppipiani}
\dfrac{\displaystyle\int_{(0,1)} \frac{c_{N,s}}{(|x|+R)^{N+2s}}\,d\mu(s)}{\displaystyle\int_{(0,1)} \frac{c_{N,s}}{(|x|-R)^{N+2s}}\,d\mu(s)} \fint_{\Omega}\tilde{u}(y)\,dy
\leq \tilde{u}(x) \leq
\dfrac{\displaystyle\int_{(0,1)} \frac{c_{N,s}}{(|x|-R)^{N+2s}}\,d\mu(s)}{\displaystyle\int_{(0,1)} \frac{c_{N,s}}{(|x|+R)^{N+2s}}\,d\mu(s)} \fint_{\Omega}\tilde{u}(y)\,dy.
\end{equation}

Now, we claim that
\begin{equation}\label{limitepiani}
\lim_{|x|\to +\infty}\left(
\dfrac{\displaystyle\int_{(0,1)} \frac{c_{N,s}}{(|x|+R)^{N+2s}}\,d\mu(s)}{\displaystyle\int_{(0,1)} \frac{c_{N,s}}{(|x|-R)^{N+2s}}\,d\mu(s)} \right) =1.
\end{equation}
To check this, first we observe that
\begin{equation}\label{troppipiani2}
\dfrac{\displaystyle\int_{(0,1)} \frac{c_{N,s}}{(|x|+R)^{N+2s}}\,d\mu(s)}{\displaystyle\int_{(0,1)} \frac{c_{N,s}}{(|x|-R)^{N+2s}}\,d\mu(s)}
=1-\dfrac{\displaystyle\int_{(0,1)} c_{N,s}\left(\frac{1}{(|x|-R)^{N+2s}}-\frac{1}{(|x|+R)^{N+2s}}\right)\,d\mu(s)}{\displaystyle\int_{(0,1)} \frac{c_{N,s}}{(|x|-R)^{N+2s}}\,d\mu(s)},
\end{equation}
where we can write
\[
\begin{aligned}
\int_{(0,1)} c_{N,s}&\left(\frac{1}{(|x|-R)^{N+2s}}-\frac{1}{(|x|+R)^{N+2s}}\right)\,d\mu(s) \\
&=\int_{(0,1)}\frac{c_{N,s}}{(|x|-R)^{N+2s}}
\left[ 1- \left(1- \frac{2R}{|x|+R}\right)^{N+2s}\right]\,d\mu(s).
\end{aligned}
\]
Setting 
\[
t:=\frac{2R}{|x|+R}\leq 1, 
\]
we have
\[
1-(1-t)^{N+2s}= (N+2s)\int_0^t (1-\tau)^{N+2s-1}\,d\tau
\leq (N+2)\int_0^t \,d\tau= (N+2)t,
\]
and so
\[
\int_{(0,1)}\frac{c_{N,s}}{(|x|-R)^{N+2s}}
\left[ 1- \left(1- \frac{2R}{|x|+R}\right)^{N+2s}\right]\,d\mu(s)
\leq \frac{2R(N+2)}{|x|+R}
\int_{(0,1)}\frac{c_{N,s}}{(|x|-R)^{N+2s}}\,d\mu(s).
\]
Accordingly,
\[
\lim_{|x|\to +\infty}
\left(\dfrac{\displaystyle\int_{(0,1)} c_{N,s}\left(\frac{1}{(|x|-R)^{N+2s}}-\frac{1}{(|x|+R)^{N+2s}}\right)\,d\mu(s)}{\displaystyle\int_{(0,1)} \frac{c_{N,s}}{(|x|-R)^{N+2s}}\,d\mu(s)}\right)
\leq \lim_{|x|\to +\infty}\frac{2R(N+2)}{|x|+R}=0.
\]
{F}rom this and~\eqref{troppipiani2} we obtain~\eqref{limitepiani}.

Thus, from~\eqref{limitepiani}, passing to the limit in~\eqref{troppipiani}, we see that
\[
\lim_{|x|\to +\infty}\tilde{u}(x)
= \frac{1}{|\Omega|}\int_\Omega \tilde{u}(x)\,dx,
\]
which implies 
\[
\lim_{|x|\to +\infty}u(x)
= \frac{1}{|\Omega|}\int_\Omega u(x)\,dx,
\]
as desired.
\end{proof}

\section{Eigenvalues and eigenfunctions}\label{xPSJODLOJDFLFVOLNF023woer}

In this section we examine the spectral properties of the problem~\eqref{prob1}. To this aim,  for any measurable function $u:\R^N\to\R$, we set
\begin{equation}\label{HMUNORM}
\|u\|_{H_\mu(\Omega)}:= \|u\|_{L^2(\Omega)} + \int_{(0, 1)} c_{N, s}\iint_{\Omega\times\Omega} \frac{|u(x)-u(y)|^2}{|x-y|^{N+2s}} dx\,dy \,d\mu(s)
\end{equation}
and we define the space
\[
H_\mu (\Omega):=\left\{u\in L^2(\Omega) : \int_{(0, 1)} c_{N, s}\iint_{\Omega\times\Omega} \frac{|u(x)-u(y)|^2}{|x-y|^{N+2s}} dx\,dy \,d\mu(s) <+\infty \right\}.
\]
Arguing as in Proposition~\ref{Hmuhilbert}, it is easy to check that $H_\mu(\Omega)$ is a Hilbert space as well. Moreover, it is worth noting that in the proof of Proposition~\ref{propcompactembedding}, we have also proved that $H_\mu(\Omega)$ is continuously embedded in $H^{s_\sharp}(\Omega)$.

In our setting, the following version of the Poincaré inequality holds:
\begin{lemma}\label{POINCARE}
Let $s\in (0, 1)$ and let $\Omega\subset\R^N$ be any bounded Lipschitz domain. 

Then, for any function $u\in H_\mu(\Omega)$, we have
\[
\int_\Omega \left|u(x)- \fint_\Omega u (x) dx \right|^2 dx \le c(\Omega, \mu) \int_{(0, 1)} c_{N,  s}\iint_{\Omega\times\Omega} \frac{|u(x)-u(y)|^2}{|x-y|^{N+2s}} dx\,dy\, d\mu(s),
\]
for a suitable constant $c(\Omega, \mu)>0$.
\end{lemma}

\begin{proof} The argument presented is a sharpening of
Lemma~3.10 in~\cite{MR3651008}.
Assume by contradiction that the desired inequality does not hold. Then, there exists a sequence $u_n\in H_\mu(\Omega)$ such that
\begin{equation}\label{av1}
\fint_\Omega u_n (x) \, dx =0, \quad \|u_n\|_{L^2(\Omega)} =1
\end{equation}
and
\begin{equation}\label{1n}
\int_{(0, 1)} c_{N, s} \iint_{\Omega\times\Omega} \frac{|u_n(x)-u_n(y)|^2}{|x-y|^{N+2s}} dx\,dy\, d\mu(s)<\frac1n.
\end{equation}
Hence, by~\eqref{HMUNORM},~\eqref{av1} and~\eqref{1n} we gather that $u_n$ is bounded in $H_\mu(\Omega)$. Besides, by Proposition~\ref{propcompactembedding}, we infer that the embeddings
\[
\mathcal H_\mu(\Omega) \subset H_\mu(\Omega) \subset H^{s_\sharp}(\Omega)
\]
are continuous. Moreover, by~\cite[Corollary 7.1]{MR2944369}, we have that~$H^{s_\sharp}(\Omega)$ is compactly embedded in $L^2(\Omega)$. Hence, up to subsequences, $u_n$ converges to some $u\in L^2(\Omega)$, namely
\[
u_n\to u \quad\mbox{ in } L^2(\Omega)\quad\mbox{ and }\quad u_n\to u \quad\mbox{ a.e. in } \Omega.
\]
Furthermore,  by~\eqref{av1},
\begin{equation}\label{av2}
\fint_\Omega u(x) \, dx =0, \quad\mbox{ and }\quad \| u\|_{L^2(\Omega)} =1.
\end{equation}
Thus, recalling that $u_n$ converges to $u$ a.e. in $\Omega$ as $n\to +\infty$, by~\eqref{1n} and Fatou's Lemma, we get 
\[
\int_{(0, 1)} c_{N, s} \iint_{\Omega\times\Omega} \frac{|u(x)-u(y)|^2}{|x-y|^{N+2s}} dx\,dy\, d\mu(s) \le\liminf_{n\to +\infty}\int_{(0, 1)} c_{N, s} \iint_{\Omega\times\Omega} \frac{|u_n(x)-u_n(y)|^2}{|x-y|^{N+2s}} dx\,dy\, d\mu(s) \le 0.
\]
Since $\mu$ is nontrivial, we infer by the previous identity that $u$ is constant in $\Omega$, in contradiction with~\eqref{av2}.
\end{proof}

We now provide the definition of eigenvalues and eigenfunctions of the operator~$L_{\alpha, \mu}$ in~\eqref{ELLE} with homogeneous $(\alpha,\mu)$-Neumann conditions in Definition~\ref{NDEFN}.

\begin{definition}
We say that $\lambda\in\R$ is an eigenvalue if the problem
\begin{equation}\label{problambda}
\begin{cases}
L_{\alpha, \mu}(u) = \lambda u \quad\mbox{ in } \Omega,\\
\mbox{with homogeneous }(\alpha,\mu)\mbox{-Neumann conditions}
\end{cases}
\end{equation}
admits a solution $u\in \mathcal H_{\alpha, \mu}(\Omega)$
which is not identically zero. 

If $\lambda$ is an eigenvalue, we call the solution $u$ an eigenfunction associated to the eigenvalue $\lambda$.
\end{definition}

With this setting, we can prove Theorem~\ref{ETGBAVA}
(in this regard, some care is needed to deal with the presence of the measure~$\mu$ and the possible coexistence with a local operator):

\begin{proof}[Proof of Theorem~\ref{ETGBAVA}]
We set
\[
L_0^2(\Omega):=\left\{f\in L^2(\Omega): \int_\Omega f(x) dx =0 \right\}.
\]
Moreover, let $T_0$ be the operator defined as
\begin{equation}\label{T0DEFN}
T_0: f\in L_0^2(\Omega)\longmapsto T_0 f = u\in\mathcal H_{\alpha, \mu}(\Omega),
\end{equation}
where $u$ denotes the unique weak solution to the problem~\eqref{prob1}, according to the Definition~\ref{weakdefn}. We point out that, by Theorem~\ref{uptoconstant} and since~$f\in L_0^2(\Omega)$, there exists a unique solution to problem~\eqref{prob1}, up to an additive constant. We can choose such a constant to be e.g. $\fint u(x) dx$. In this way,
\[
\int_\Omega \left(u(x)-\fint_\Omega u(x) dx\right) dx =0
\]
and then $u\in L_0^2(\Omega)$.

Also, we define the operator $T$ as the restriction
\begin{equation}\label{TRESTRICTION}
T:f\in L_0^2(\Omega)\longmapsto T_0 f \big|_\Omega \in L_0^2(\Omega).
\end{equation}
Let us show that $T$ is a compact and self-adjoint operator.

We start proving that $T$ is compact. To this,  we test~\eqref{weaksolution} with $v= u = T_0 f$, getting
\begin{equation}\label{case12}
\begin{split}
\alpha\int_\Omega |\nabla u|^2 dx +\int_{(0, 1)} \frac{c_{N, s}}{2} \iint_{\Q} \frac{|u(x)-u(y)|^2}{|x-y|^{N+2s}} dx\,dy\, d\mu(s)& = \int_\Omega f u \, dx\\
&\le \|f\|_{L^2(\Omega)}\|u\|_{L^2(\Omega)}.
\end{split}
\end{equation}
We notice that two cases can occur:
\begin{itemize}
\item if $\alpha\neq 0$,  by~\eqref{case12}, we infer that
\begin{equation}\label{ALPHANEQ0}
\begin{split}
\alpha\int_\Omega |\nabla u|^2 dx &\le \alpha\int_\Omega |\nabla u|^2 dx +\int_{(0, 1)} \frac{c_{N, s}}{2} \iint_{\Q} \frac{|u(x)-u(y)|^2}{|x-y|^{N+2s}} dx\,dy\, d\mu(s)\\
&\le \|f\|_{L^2(\Omega)}\|u\|_{L^2(\Omega)}.
\end{split}
\end{equation}
Also, the classical Poincaré inequality gives
\begin{equation}\label{POINCARECL}
\|u\|_{L^2(\Omega)}\le c_1 \left(\int_\Omega |\nabla u|^2 dx\right)^{1/2},
\end{equation}
for a suitable constant $c_1>0$ depending only on $\Omega$.  Thus, combining~\eqref{ALPHANEQ0} and~\eqref{POINCARECL}, we have
\begin{equation}\label{ALPHA1}
\alpha \left(\int_\Omega |\nabla u|^2 dx\right)^{1/2}\le c_2 \|f\|_{L^2(\Omega)},
\end{equation}
for a suitable $c_2>0$;
\item if $\alpha = 0$, by~\eqref{case12} it follows that
\begin{equation}\label{ALPHA=0}
\begin{split}
\int_{(0, 1)} \frac{c_{N, s}}{2} \iint_{\Omega\times\Omega} \frac{|u(x)-u(y)|^2}{|x-y|^{N+2s}} dx\,dy\, d\mu(s)&\le \int_{(0, 1)} \frac{c_{N, s}}{2} \iint_{\Q} \frac{|u(x)-u(y)|^2}{|x-y|^{N+2s}} dx\,dy\, d\mu(s)\\
& \le \|f\|_{L^2(\Omega)}\|u\|_{L^2(\Omega)}.
\end{split}
\end{equation}
Now, recalling that $\fint u (x) dx =0$ since $u\in L_0^2(\Omega)$,  by Lemma~\ref{POINCARE} we get
\[
\|u\|_{L^2(\Omega)} \le \left(c(\Omega, \mu)\int_{(0, 1)} \frac{c_{N, s}}{2} \iint_{\Omega\times\Omega} \frac{|u(x)-u(y)|^2}{|x-y|^{N+2s}} dx\,dy\, d\mu(s)\right)^{1/2}.
\]
Thus, ~\eqref{ALPHA=0} becomes
\begin{equation}\label{ALPHA2}
\left(\int_{(0, 1)} \frac{c_{N, s}}{2} \iint_{\Omega\times\Omega} \frac{|u(x)-u(y)|^2}{|x-y|^{N+2s}} dx\,dy\, d\mu(s)\right)^{1/2} \le c_3 \|f\|_{L^2(\Omega)},
\end{equation}
with $c_3>0$ depending only on $\Omega$ and $\mu$.
\end{itemize}
Now, let $f_n$ be any bounded sequence in $L_0^2(\Omega)$ and let $u_n:= T f_n$. On the one hand, if $\alpha\neq 0$, by~\eqref{ALPHA1} we obtain that $u_n$ is bounded in $H^1(\Omega)$, which is compactly embedded in $L^2(\Omega)$.  Hence, $u_n$ converges strongly to some $u\in L^2(\Omega)$, up to subsequences. 

On the other hand, if $\alpha = 0$, by~\eqref{ALPHA2} we infer that $u_n$ is bounded in~$H_\mu(\Omega)$. By~\cite[Corollary 7.2]{MR2944369}, the space~$H_\mu(\Omega)$ is compactly embedded in $L^2(\Omega)$ and again $u_n$ converges strongly, up to subsequences, to some $u\in L^2(\Omega)$. Consequently, we have that $T$ is a compact operator.

Let us show that $T$ is self-adjoint. To this aim, we take $f_1, f_2\in C_0^\infty(\Omega)$ satisfying
\[
\fint_\Omega f_1(x) dx = \fint_\Omega f_2(x) dx =0.
\]
Hence, by~\eqref{T0DEFN} and~\eqref{weaksolution} (recall that $g\equiv h\equiv 0$), we have
\begin{equation}\label{F1}
\begin{split}
\alpha\int_\Omega \nabla (T_0 f_1)(x)\cdot \nabla v(x) \,dx + &\int_{(0, 1)} \frac{c_{N, s}}{2} \iint_{\Q} \frac{(T_0 f_1(x)-T_0 f_1(y))(v(x) -v(y))}{|x-y|^{N+2s}} \,dx\,dy\, d\mu(s)\\
& = \int_\Omega f_1 v \,dx \qquad\mbox{ for any } v\in \mathcal H_{\alpha, \mu}(\Omega)
\end{split}
\end{equation}
and 
\begin{equation}\label{F2}
\begin{split}
\alpha\int_\Omega \nabla (T_0 f_2)(x)\cdot \nabla w(x)\, dx + &\int_{(0, 1)} \frac{c_{N, s}}{2} \iint_{\Q} \frac{(T_0 f_2(x)-T_0 f_2(y))(w(x) -w(y))}{|x-y|^{N+2s}}\, dx\,dy\, d\mu(s)\\
& = \int_\Omega f_2 w\, dx \qquad\mbox{ for any } w\in \mathcal H_{\alpha, \mu}(\Omega).
\end{split}
\end{equation}
Now, recalling the definition~\eqref{TRESTRICTION} and testing~\eqref{F1} and~\eqref{F2} with~$v:= T_0 f_2\in \mathcal H_{\alpha, \mu}(\Omega)$ and~$w:= T_0 f_1\in \mathcal H_{\alpha, \mu}(\Omega)$ respectively, we obtain that
\begin{equation}\label{F12}
\int_\Omega f_1(x) \, T f_2(x) dx = \int_\Omega f_2(x) \, T f_1(x) dx \quad\mbox{ for any } f_1, f_2\in C_0^\infty(\Omega).
\end{equation}
We employ a density argument to prove that the identity~\eqref{F12} holds for any $f_1, f_2\in L_0^2(\Omega)$. For this, we notice that if $f_1, f_2\in L_0^2(\Omega)$, one can find two sequences $f_{1, n}, f_{2, n}\in C_0^\infty(\Omega)$ such that $f_{1, n}\to f_1$ in $L^2(\Omega)$ and $f_{2, n}\to f_2$ in $L^2(\Omega)$, as $n\to +\infty$. 

Thus, by~\eqref{F12} we deduce that
\begin{equation}\label{F12K}
\int_\Omega f_{1, n} (x) \, T f_{2, n}(x) dx = \int_\Omega f_{2, n}(x) \, T f_{1, n}(x) dx.
\end{equation}
Moreover, recalling that~\eqref{ALPHA1} and~\eqref{ALPHA2} hold with $u=T_0 f$ and $f\in L^2(\Omega)$,  either by~\eqref{TRESTRICTION}, \eqref{POINCARECL}, \eqref{ALPHA1} if~$\alpha\neq 0$ or by Lemma~\ref{POINCARE} and~\eqref{ALPHA2} if~$\alpha =0$, we have
\[
\|T f_{1, n}\|_{L^2(\Omega)}\le c_4 \|f_{1, n}\|_{L^2(\Omega)} \quad\mbox{ and }\quad \|T f_{2, n}\|_{L^2(\Omega)}\le c_5 \|f_{2, n}\|_{L^2(\Omega)},
\]
for some constants $c_4, c_5>0$.

Thus, $T f_{1, n}\to T f_1$ in $L^2(\Omega)$ and $T f_{2, n}\to T f_2$ in $L^2(\Omega)$ as $n\to +\infty$ and by~\eqref{F12K} we have
\[
\int_\Omega f_1 (x) \, T f_2(x) dx = \int_\Omega f_2(x) \, T f_1(x) dx,
\]
which means that $T$ is a self-adjoint operator in $L^2(\Omega)$.

Then, by the spectral theorem, there exists a sequence of eigenvalues~$\{\mu_i\}_{i\ge 2}$ of~$T$ with corresponding eigenfunctions~$\{e_i\}_{i\ge 2}$ which are a complete orthogonal system of~$L_0^2(\Omega)$.

We claim that
\begin{equation}\label{MUNEQ0}
\mu_i \neq 0 \quad\mbox{ for any } i\in\{2,\dots, n\}.
\end{equation}
If not, there exists $i\in\{2, \dots, n\}$ such that $\mu_i=0$. Hence,
\begin{equation}\label{0MUI}
0=\mu_i e_i = T e_i = T_0 e_i \quad\text{ in } \Omega
\end{equation}
and also 
\[
\int_{(0,1)} \Ns e_i(x)\,d\mu(s)=0 
\]
in $\R^n\setminus\overline{\Omega}$, by construction. 

This, along with~\eqref{0MUI}, imply that
\[
T_0 e_i(x) =\frac{\displaystyle \int_{(0, 1)}c_{N,s}\int_\Omega \frac{T_0 e_i(y)}{|x-y|^{n+2s}} dy \, d\mu(s)}{\displaystyle \int_{(0, 1)}c_{N,s}\int_\Omega \frac{1}{|x-y|^{n+2s}} dy\, d\mu(s)} =0 \quad\text{ in } \R^n\setminus\overline{\Omega}.
\]
Hence, by the previous identity and~\eqref{0MUI} we infer that
\[
0= -\alpha \Delta (T_0 e_i) + \int_{(0, 1)} (-\Delta)^s (T_0 e_i)\, d\mu(s) =e_i \quad\mbox{ in } \Omega
\]
in the weak sense, which gives $e_i\equiv 0$ in $\Omega$. This contradicts the fact that $e_i$ is an egenfunction and then~\eqref{MUNEQ0} holds.

Now, in light~\eqref{MUNEQ0}, we set
\[
\lambda_i:= \mu_i^{-1} \quad\text{ for any } i\in\{1,\dots, n\}.
\]
Moreover, we also define $u_i:= T_0 e_i$ and we claim that $u_2, u_3, \dots$ is the desired systems of eigenfunctions with corresponding eigenvalues $\lambda_2, \lambda_3,\dots$.

Indeed, by definition
\[
u_i = T_0 e_i = T e_i = \mu_i e_i\quad\mbox{ in } \Omega,
\]
which means that the desired properties of orthogonality and completeness properties of $u_2, u_3, \dots$ simply follows by those of $e_2, e_3, \dots$ and the previous identity provides that
\[
-\alpha \Delta (u_i) + \int_{(0, 1)} (-\Delta)^s (u_i)\, d\mu(s) =-\alpha \Delta (T_0 e_i)+ \int_{(0, 1)} (-\Delta)^s (T_0 e_i)\, d\mu(s) =e_i =\lambda_i u_i \quad\mbox{ in } \Omega,
\]
which means that also the spectral property holds. Let us prove that
\begin{equation}\label{lambda>0}
\lambda_i >0 \quad\mbox{ for any } i\ge 2.
\end{equation}
To this end, we notice that the corresponding eigenfunctions $u_1, u_2, \dots$ are solutions of  the problem
\[
\begin{cases}
L_{\alpha,\mu}(u_i)=\lambda_i u_i  \quad \mbox{ in } \Omega
\\
\mbox{with homogeneous } (\alpha,\mu)\mbox{-Neumann conditions},
\end{cases}
\]
Then, taking $u_i$ as a test function in the weak formulation of this problem, we find
\begin{equation*}
\alpha \int_\Omega |\nabla u(x)|^2dx +\int_{(0, 1)} \frac{c_{N, s}}{2} \iint_{\Q} \frac{|u(x)-u(y)|^2}{|x-y|^{N+2s}} dx\,dy\, d\mu(s)= \lambda_i \int_\Omega u_i(x)^2 dx,
\end{equation*}
which entails that $\lambda_i\ge 0$.  

Let us prove that~$\lambda_i>0$. To check this, suppose by contradiction that~$\lambda_i=0$. Then, Lemma~\ref{lemmasoluzionecostante} provides that~$u_i$ is constant. As we also know that that~$u_i\in L^2_0(\Omega)$, it follows that~$u_i\equiv 0$, against the fact that~$u_i$ is an eigenvalue.  This proves~\eqref{lambda>0}.

{F}rom~\eqref{lambda>0}, up to reordering them, we can suppose that $0<\lambda_2\le\lambda_3\le\dots$ and we
notice that $\lambda_1=0$ is an eigenvalue, with eigenfunction $u_1= 1$,  again by Lemma~\ref{lemmasoluzionecostante}.
Therefore, we have a sequence of eigenvalues $0 =\lambda_1 <\lambda_2\le\lambda_3\le\dots$ and its corresponding eigenfunctions are a complete orthogonal system in~$L^2(\Omega)$ (see the proof of~\cite[Theorem 3.11 ]{MR3651008}).
This concludes the proof.
\end{proof}

\section{The heat equation}\label{KDMCLDFMVV2}

In this section we study the heat equation with homogeneous
Neumann conditions, namely
\begin{equation}\label{problemaequazionecalore}
\begin{cases}
\partial_t u(x,t)+L_{\alpha,\mu}(u(x,t))=0  & \mbox{in }
\Omega\times (0,+\infty),
\\
u(x,0)= u_0(x)  & \mbox{in } \Omega,
\\
\mbox{with homogeneous }(\alpha,\mu)\mbox{-Neumann conditions}
& \mbox{in } (0,+\infty).
\end{cases}
\end{equation} 
Here, we deal with classical solutions, 
that is $u(\cdot,t)\in L^\infty(\R^N)\cap C(\overline{\Omega})\cap C^2(\Omega)$ for any~$t\in [0,+\infty)$ and~$u(x,\cdot)\in C^1((0,+\infty))\cap C([0,+\infty))$,
with the homogeneous $(\alpha,\mu)$-Neumann conditions satisfied 
pointwise in~$\R^N\setminus \Omega$ and~$\partial \Omega$ 
respectively.

The next two results give that classical solutions of problem
\eqref{problemaequazionecalore} preserve their mass and have 
energy that decreases in time. 

We point out that the assumptions required throughout this section are significantly more general than the ones required in \cite[Propositions 4.1, 4.2 and 4.3]{MR3651008}.

\begin{proposition}\label{proposizionemassacostante}
Let $u(x,t)$ be a classical solution of
\eqref{problemaequazionecalore} satisfying
\begin{equation}\label{ipotesimassa1}
\partial_t u \in L^\infty_{loc}\left((0,+\infty),L^1(\Omega)\right)
\end{equation}
and
\begin{equation}\label{ipotesimassa2}
\alpha\Delta u(\cdot,t)\in L^1(\Omega)\quad \mbox{and}\quad(-\Delta)^s u(\cdot,t)\in L^1\left(\Omega\times (0,1)\right)\quad\mbox{ for any } t\in (0,+\infty).
\end{equation}

Assume that the following statements hold:
\begin{itemize}
\item if $\alpha\neq 0$, then
\begin{equation}\label{ipotesimassa3}
u(\cdot,t)\in C^1(\overline{\Omega})\quad\mbox{ for any } t\in (0,+\infty);
\end{equation}
\item if $\mu \not \equiv 0$, then, for any $t\in (0,+\infty)$, the function
\begin{equation}\label{ipotesimassa4}
\Omega\times (\R^N \setminus \Omega)\times (0,1)\ni (x,y,s)\mapsto
\frac{u(x,t)-u(y,t)}{|x-y|^{N+2s}} \mbox{ belongs to }
L^1\left(\Omega\times (\R^N \setminus \Omega)\times (0,1)\right).
\end{equation}
\end{itemize}

Then, 
\[
\int_\Omega u(x,t)\,dx =\int_\Omega u_0(x)\,dx\quad\mbox{ for any } t\in (0,+\infty),
\]
namely the total mass is conserved.
\end{proposition}

\begin{proof} The argument presented here is a careful refinement of Proposition~4.1 in~\cite{MR3651008}.
In light of~\eqref{ipotesimassa1} and~\eqref{ipotesimassa2},
from the dominated convergence theorem we have
\[
\frac{d}{dt}\int_\Omega u(x,t)\,dx
=\int_\Omega \partial_t u(x,t)\,dx
=\alpha \int_\Omega \Delta u(x,t)\,dx
-\int_{(0,1)}\int_\Omega(-\Delta)^s u(x,t)\,dx\,d\mu(s).
\]
Moreover, by~\eqref{ipotesimassa3} and~\eqref{ipotesimassa4}, and using the homogeneous Neumann conditions in Definition~\ref{NDEFN},
we have that divergence theorem and Lemma~\ref{lemmathdivergenza} apply and
\[
\frac{d}{dt}\int_\Omega u(x,t)\,dx
=\alpha\int_{\partial \Omega} \partial_\nu u(x)\,d\mathscr{H}_x^{N-1}
+\int_{(0,1)}\int_{\R^N\setminus \Omega} \Ns u(x,t)\,dx\,d\mu(s)
=0.
\]
This implies that the quantity $\int_\Omega u(x,t)\,dx$ does 
not depend on $t$, which concludes the proof.
\end{proof}

\begin{proposition}\label{DECENSD}
Let $u(x,t)$ be a classical solution of
\eqref{problemaequazionecalore} such that 
\begin{equation}\label{ipotesienergia1}
\alpha \nabla u(\cdot,t)\nabla \partial_t u(\cdot,t)
\in L^1(\Omega)
\end{equation}and suppose that
the function
\begin{equation}\label{ipotesienergia2}\begin{split}&
\Q\times (0,1)\ni (x,y,s) \mapsto c_{N,s}
\frac{(u(x,t)-u(y,t))(\partial_t u(x,t)-\partial_t u(y,t))}{|x-y|^{N+2s}} \\&\mbox{belongs to } L^1(\Q\times (0,1)),\end{split}
\end{equation}
and
\begin{equation}\label{ipotesienergia3}
\alpha \partial_t u(\cdot,t)\Delta u(\cdot,t)\in L^1(\Omega),\quad\partial_t u(\cdot,t)(-\Delta)^s u(\cdot,t)\in L^1(\Omega\times (0,1)),
\end{equation}
for any $t\in (0,+\infty)$.

In addition, assume that:
\begin{itemize}
\item if $\alpha\neq 0$, then the assumption in~\eqref{ipotesimassa3} is satisfied,
\item if $\mu \not \equiv 0$, then, for any $t\in (0,+\infty)$, the function
\begin{equation}\label{ipotesienergia4}
\begin{split}&
(\R^N \setminus \Omega)\times \Omega \times (0,1)\ni (x,y,s)\mapsto
c_{N,s}\frac{(u(x,t)-u(y,t))\partial_t u(x,t)}{|x-y|^{N+2s}} \\
&\mbox{ belongs to }
L^1\left((\R^N \setminus \Omega)\times \Omega \times (0,1)\right).
\end{split}
\end{equation}
\end{itemize}

Then, the energy
\[
E(t):=\frac{\alpha}{2}\int_\Omega |\nabla u(x,t)|^2\,dx
+\int_{(0,1)}\frac{c_{N,s}}{4}
\iint_\Q \frac{|u(x,t)-u(y,t)|^2}{|x-y|^{N+2s}}\,dx\,dy\,d\mu(s)
\]
is decreasing in time $t>0$ (and strictly decreasing, unless~$u$ is independent of time).
\end{proposition}

\begin{proof} This proof is inspired to Proposition~4.2 in~\cite{MR3651008}, but here some additional care is required to distinguish the cases~$\alpha\ne0$ and~$\mu\not\equiv0$.

We compute $E'(t)$ and show that it is negative.
Indeed, recalling~\eqref{ipotesienergia1} and
\eqref{ipotesienergia2}, we have
\begin{equation}\label{equazioneE'}
\begin{aligned}
E'(t)&=\alpha\int_\Omega \nabla u(x,t)\nabla \partial_t u(x,t)\,dx 
\\
&\qquad+\int_{(0,1)}\frac{c_{N,s}}{2}
\iint_\Q \frac{(u(x,t)-u(y,t))(\partial_t u(x,t)-\partial_t u(y,t))}{|x-y|^{N+2s}}\,dx\,dy\,d\mu(s).
\end{aligned}
\end{equation}
In light of~\eqref{ipotesimassa3}, \eqref{ipotesienergia1},
\eqref{ipotesienergia2}, \eqref{ipotesienergia3} and~\eqref{ipotesienergia4}, from Lemma~\ref{lemmaperparti}, the 
classical integration by parts formula and the homogeneous Neumann
conditions, we gather that
\[
\begin{aligned}
\alpha\int_\Omega &\nabla u(x,t)\nabla \partial_t u(x,t)\,dx 
+\int_{(0,1)}\frac{c_{N,s}}{2}
\iint_\Q \frac{(u(x,t)-u(y,t))(\partial_t u(x,t)-\partial_t u(y,t))}{|x-y|^{N+2s}}\,dx\,dy\,d\mu(s) \\
&=-\alpha \int_\Omega \partial_t u(x,t) \Delta u(x,t)\,dx
+\int_{(0,1)} \int_\Omega \partial_t u(x,t) (-\Delta)^s u(x,t)\,dx\,d\mu(s) \\
&=- \int_\Omega \partial_t u(x,t)
\left(\alpha\Delta u(x,t)\,dx -\int_{(0,1)}(-\Delta)^s u(x,t)\,d\mu(s)\right) \,dx.
\end{aligned}
\]
Thus,~\eqref{equazioneE'} becomes
\[
E'(t)=- \int_\Omega \partial_t u(x,t)
\left(-L_{\alpha,\mu}u(x,t) \right)\,dx
=- \int_\Omega |\partial_t u(x,t)|^2\,dx\leq 0,
\]
with the equality holding if and only if $u$ is constant in time.  This concludes the proof.
\end{proof}

The next result generalizes \cite[Proposition 4.3]{MR3651008} to our setting, with the proof postponed to Appendix \ref{appendixASUYY}.

\begin{proposition}\label{ASUYY}
Let $u(x,t)$ be a classical solution of
\eqref{problemaequazionecalore} satisfying~\eqref{ipotesimassa1}
and~\eqref{ipotesimassa2}.

In addition, assume that:
\begin{itemize}
\item if $\alpha\neq 0$, then the assumption in~\eqref{ipotesimassa3} is satisfied,
\item if $\mu \not \equiv 0$, then, for any $t\in (0,+\infty)$, the assumption in~\eqref{ipotesimassa4} is satisfied.
\end{itemize}

Then,
\[
u(x,t) \to \frac{1}{|\Omega|}
\int_\Omega u_0(x)\,dx \mbox{ in } L^2(\Omega) \quad\mbox{ as } t\to +\infty.
\]
\end{proposition}

\section{Continuity properties}\label{KDMCLDFMVV3}

In this section we study some continuity properties induced by the 
nonlocal Neumann condition~\eqref{Neumannomogeneo} and we prove
Theorem~\ref{propcontinuità}.

\begin{proof}[Proof of Theorem~\ref{propcontinuità}]
First, we fix $x_0\in \R^N\setminus \overline{\Omega}$ and we 
prove that~$u$ is continuous at~$x_0$. Since~$\R^N\setminus \overline{\Omega}$ is an open set, there exists~$\rho>0$ such that~$|x_0-y|\geq \rho$ for any~$y\in \Omega$.
So, if~$x\in B_{\rho/2}(x_0)$, we have that
\[
|x-y|\geq |x_0-y|-|x_0-x|\geq \frac{\rho}{2}\quad\mbox{ for any } y\in\Omega.
\]
As a concequence, if $\rho$ is small enough,
for any $x\in B_{\rho/2}(x_0)$ and $y\in \Omega$
we have that 
\[
\frac{|u(y)|+1}{|x-y|^{N+2s}}
\leq \frac{2^{N+2s}}{\rho^{N+2s}}\Big(\|u\|_{L^\infty(\overline{\Omega})}+1\Big)
\leq \frac{2^{N+2}}{\rho^{N+2}}\Big(\|u\|_{L^\infty(\overline{\Omega})}+1\Big).
\]
{F}rom this, we deduce that 
\[
\begin{split}
\int_{(0,1)}c_{N,s}\int_\Omega \frac{|u(y)|+1}{|x-y|^{N+2s}}\,dy\,d\mu(s)
&\leq \frac{2^{N+2}}{\rho^{N+2}}|\Omega|\Big(\|u\|_{L^\infty(\overline{\Omega})}+1\Big)\int_{(0,1)}c_{N,s}\,d\mu(s)\\
&\le\frac{2^{N+2}}{\rho^{N+2}}|\Omega|\Big(\|u\|_{L^\infty(\overline{\Omega})}+1\Big) \sup_{s\in (0, 1)} c_{N, s} \int_{(0, 1)} d\mu(s)\\
&<+\infty.
\end{split}
\]
Thus, by the Neumann condition in~\eqref{NNN} and the dominated convergence theorem we obtain that
\[
\begin{split}
\lim_{x\to x_0}u(x) &=
\lim_{x\to x_0} \dfrac{\displaystyle \int_{(0,1)}c_{N,s}\int_\Omega \dfrac{u(y)}{|x-y|^{N+2s}}\,dy\,d\mu(s)}{\displaystyle\int_{(0,1)}c_{N,s}\int_\Omega \dfrac{1}{|x-y|^{N+2s}}\,dy\,d\mu(s)}\\
&=\dfrac{\displaystyle \int_{(0,1)}c_{N,s}\int_\Omega \dfrac{u(y)}{|x_0-y|^{N+2s}}\,dy\,d\mu(s)}{\displaystyle\int_{(0,1)}c_{N,s}\int_\Omega \dfrac{1}{|x_0-y|^{N+2s}}\,dy\,d\mu(s)}\\
&=u(x_0).
\end{split}
\]
This proves that $u$ is continuous at any point of 
$\R^N\setminus \overline{\Omega}$.

Now we show that $u$ is continuous at any point 
$p\in \partial\Omega$.
To do so, we need a more accurate argument, since numerators and denominators
may create singularities which have to be attentively estimated to detect cancellations.
We take a sequence $\{p_k\}_k$
converging to $p$ as $k$ goes to infinity.
Let $q_k$ be the projection of $p_k$ onto $\overline{\Omega}$.
Since $p\in \overline{\Omega}$, from the minimizing property of the 
projection we have that 
\[
|p_k-q_k|=\inf_{\xi\in \overline{\Omega}}|p_k-\xi|
\leq |p_k-p|,
\]
and so
\[
\lim_{k\to +\infty}|q_k-p|
\leq \lim_{k\to +\infty}\big(|q_k-p_k|+|p_k-p|\big)
\leq \lim_{k\to +\infty} 2|p_k-p|=0.
\]
Thus, since by assumption $u$ is continuous in $\overline{\Omega}$,
we have
\begin{equation}\label{pkconvergenza1}
\lim_{k\to +\infty}u(q_k)=u(p).
\end{equation}
Now we claim that
\begin{equation}\label{pkconvergenza2}
\lim_{k\to +\infty}\big(u(q_k)-u(p_k)\big)=0.
\end{equation}
To prove this, it is enough to consider the points $p_k$ that belong
to $\R^N \setminus\overline{\Omega}$. Indeed, if 
$p_k\in \overline{\Omega}$, we have $p_k=q_k$ and so 
\eqref{pkconvergenza2} is trivially satisfied.

We define $\nu_k:=(p_k-q_k)/|p_k-q_k|$, so that $\nu_k$ is the 
exterior normal of $\Omega$ at $q_k\in \partial\Omega$.
We consider a rigid motion~$\mathcal{R}_k$ such that~$\mathcal{R}_k q_k=0$ and~$\mathcal{R}_k \nu_k=e_N=(0,\cdots,0,1)$ (see Figure~\ref{RK}).  Let~$h_k:=|p_k-q_k|$. We notice that
\begin{equation}\label{dovevapk}
h_k^{-1}\mathcal{R}_k p_k=
h_k^{-1}\mathcal{R}_k (p_k-q_k)=\mathcal{R}_k \nu_k=e_N.
\end{equation}
Then, the domain
\[
\Omega_k:=h_k^{-1}\mathcal{R}_k \Omega
\]
has a vertical exterior normal at 0 and approaches the half space
$\Pi:=\{x_N<0\}$ as $k\to \infty$.

%\vspace{-5em}
\begin{figure}[h]
\begin{center}
\includegraphics[scale=.42]{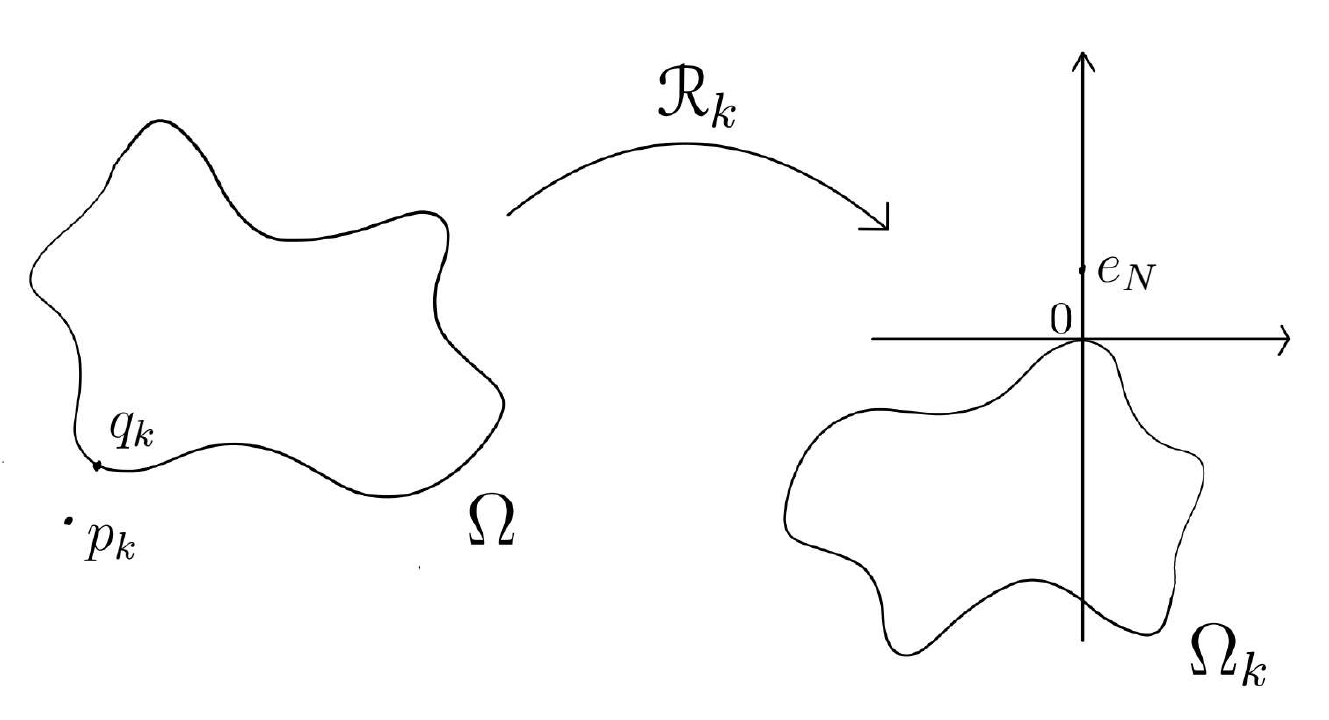}
\end{center}
\caption{The rigid motion $\mathcal R_k$ of $\Omega$.
%%%% when the domain $\Omega_k$ is Australia :)
}
\label{RK}
\end{figure}

Now, we use the homogeneous Neumann condition at $p_k$ to obtain
\[
\begin{split}
u(p_k)-u(q_k)&=\dfrac{\displaystyle\int_{(0,1)}c_{N,s}\int_{\Omega}\dfrac{u(y)}{|p_k-y|^{N+2s}}\,dy\,d\mu(s)}{\displaystyle\int_{(0,1)}c_{N,s}\int_{\Omega}\dfrac{1}{|p_k-y|^{N+2s}}\,dy\,d\mu(s)}-u(q_k)\\
&=\dfrac{\displaystyle\int_{(0,1)}c_{N,s}\int_{\Omega}\dfrac{u(y)-u(q_k)}{|p_k-y|^{N+2s}}\,dy\,d\mu(s)}{\displaystyle\int_{(0,1)}c_{N,s}\int_{\Omega}\dfrac{1}{|p_k-y|^{N+2s}}\,dy\,d\mu(s)}\\
&=I_1+I_2,
\end{split}
\]
with 
\[
I_1:=\dfrac{\displaystyle\int_{(0,1)}c_{N,s}\int_{\Omega\cap B_{\sqrt{h_k}}(q_k)}\dfrac{u(y)-u(q_k)}{|p_k-y|^{N+2s}}\,dy\,d\mu(s)}{\displaystyle\int_{(0,1)}c_{N,s}\int_{\Omega}\dfrac{1}{|p_k-y|^{N+2s}}\,dy\,d\mu(s)}
\]
and
\begin{equation}\label{I2}
I_2:=\dfrac{\displaystyle\int_{(0,1)}c_{N,s}\int_{\Omega\setminus B_{\sqrt{h_k}}(q_k)}\dfrac{u(y)-u(q_k)}{|p_k-y|^{N+2s}}\,dy\,d\mu(s)}{\displaystyle\int_{(0,1)}c_{N,s}\int_{\Omega}\dfrac{1}{|p_k-y|^{N+2s}}\,dy\,d\mu(s)}.
\end{equation}
We observe that from the uniform continuity of $u$ in $\Omega$, we
have
\[
\lim_{k\to +\infty}\sup_{\Omega\cap B_{\sqrt{h_k}}(q_k)}
|u(y)-u(q_k)|=0.
\]
Therefore,
\begin{equation}\label{I1to0}
\lim_{k\to +\infty} |I_1|\le\lim_{k\to +\infty} \, \sup_{\Omega\cap B_{\sqrt{h_k}}(q_k)} |u(y)-u(q_k)|= 0.
\end{equation}
Moreover, using the change of variable 
$\eta:=h_k^{-1}\mathcal{R}_k y$ and recalling~\eqref{dovevapk},
we obtain
\[
\begin{aligned}
|I_2|&\leq \ \dfrac{\displaystyle\int_{(0,1)}c_{N,s}\int_{\Omega\setminus B_{\sqrt{h_k}}(q_k)}\dfrac{|u(y)-u(q_k)|}{|p_k-y|^{N+2s}}\,dy\,d\mu(s)}{\displaystyle\int_{(0,1)}c_{N,s}\int_{\Omega}\dfrac{1}{|p_k-y|^{N+2s}}\,dy\,d\mu(s)}\\
&\leq 2\|u\|_{L^\infty(\overline{\Omega})} \ \dfrac{\displaystyle\int_{(0,1)}c_{N,s}\int_{\Omega\setminus B_{\sqrt{h_k}}(q_k)}\dfrac{1}{|p_k-y|^{N+2s}}\,dy\,d\mu(s)}{\displaystyle\int_{(0,1)}c_{N,s}\int_{\Omega}\dfrac{1}{|p_k-y|^{N+2s}}\,dy\,d\mu(s)} \\
&=2\|u\|_{L^\infty(\overline{\Omega})} \ \dfrac{\displaystyle\int_{(0,1)}c_{N,s}\int_{\Omega_k\setminus B_{1/{\sqrt{h_k}}}}\dfrac{h_k^N}{|h_k\mathcal{R}_k^{-1}e_N-h_k\mathcal{R}_k^{-1}\eta|^{N+2s}}\,d\eta\,d\mu(s)}{\displaystyle\int_{(0,1)}c_{N,s}\int_{\Omega_k}\dfrac{h_k^N}{|h_k\mathcal{R}_k^{-1}e_N-h_k\mathcal{R}_k^{-1}\eta|^{N+2s}}\,d\eta\,d\mu(s)} \\
&=2\|u\|_{L^\infty(\overline{\Omega})} \ \dfrac{\displaystyle\int_{(0,1)}c_{N,s}\int_{\Omega_k\setminus B_{1/{\sqrt{h_k}}}}\dfrac{h_k^{-2s}}{|e_N-\eta|^{N+2s}}\,d\eta\,d\mu(s)}{\displaystyle\int_{(0,1)}c_{N,s}\int_{\Omega_k}\dfrac{h_k^{-2s}}{|e_N-\eta|^{N+2s}}\,d\eta\,d\mu(s)}.
\end{aligned}
\]
Now, we observe that, if 
$\eta \in \Omega_k\setminus B_{1/{\sqrt{h_k}}}$, we have
\[
\begin{aligned}&
|e_N-\eta|^{N+2s}=|e_N-\eta|^{N+s}|e_N-\eta|^s
\geq |e_N-\eta|^{N+s} (|\eta|-1)^s \\
&\qquad\qquad=|e_N-\eta|^{N+s} (h_k^{-1/2}-1)^s
\geq |e_N-\eta|^{N+s} h_k^{-s/4},
\end{aligned}
\]
for $k$ large enough, and consequently
\begin{equation}\label{I2minoreuguale}
|I_2|\leq 2\|u\|_{L^\infty(\overline{\Omega})} \ \dfrac{\displaystyle\int_{(0,1)}c_{N,s}\int_{\Omega_k}\dfrac{h_k^{-\frac{7}{4}s}}{|e_N-\eta|^{N+2s}}\,d\eta\,d\mu(s)}{\displaystyle\int_{(0,1)}c_{N,s}\int_{\Omega_k}\dfrac{h_k^{-2s}}{|e_N-\eta|^{N+2s}}\,d\eta\,d\mu(s)}.
\end{equation}
Then, setting
\begin{equation}\label{I2mMM}
m:=\min\left\{n\in \N:\, \mbox{supp}(\mu) \cap\Bigg(\left(\frac{7}{8}\right)^n,\left(\frac{7}{8}\right)^{n-1}\Bigg]\neq \emptyset\right\},
\end{equation}
we can multiply and divide~\eqref{I2minoreuguale} by $h_k^{2(7/8)^m}$ to obtain
\begin{equation}\label{I2minoreuguale7/8}
|I_2|\leq 2\|u\|_{L^\infty(\overline{\Omega})}
\dfrac{\displaystyle\int_{(0,(7/8)^{m-1}]}c_{N,s}\int_{\Omega_k}\dfrac{h_k^{2\left(\frac{7}{8}\right)^m-\frac{7}{4}s}}{|e_N-\eta|^{N+2s}}\,d\eta\,d\mu(s)}{\displaystyle\int_{(0,(7/8)^{m-1}]}c_{N,s}\int_{\Omega_k}\dfrac{h_k^{2\left(\frac{7}{8}\right)^m-2s}}{|e_N-\eta|^{N+2s}}\,d\eta\,d\mu(s)}.
\end{equation}
Recalling that $\Omega_k$ converges to the half space $\Pi$, if~$k$ is large enough, then we can suppose that~$\Omega_k \subset \R^N\setminus B_{1/2}(e_N)$. In this way, 
\[
\begin{aligned}
c_{N,s}&\int_{\Omega_k}\dfrac{1}{|e_N-\eta|^{N+2s}}\,d\eta\,d\mu(s)
\leq c_{N,s}\int_{\R^N\setminus B_{1/2}(e_N)}\dfrac{1}{|e_N-\eta|^{N+2s}}\,d\eta\,d\mu(s) \\
&\quad=c_{N,s} \,\omega_{N-1}\int_{1/2}^{+\infty}\rho^{-1-2s}\,d\rho
=\frac{c_{N,s}}{s}\,2^{2s-1}\omega_{N-1},
\end{aligned}
\]
which is uniformly bounded for $s\in (0,1)$. 

Moreover, in view of~\eqref{I2mMM} we can focus on the case $s<(7/8)^{m-1}$
(otherwise we exit the support of~$\mu$).
As a result, we see that $$h_k^{2\left(\frac{7}{8}\right)^m-\frac{7}{4}s}\to 0
\quad{\mbox{ as $k\to +\infty$}}$$ and therefore, by the
dominated convergence theorem,
\begin{equation}\label{I_2numeratore}
\lim_{k\to+\infty}
\displaystyle\int_{(0,(7/8)^{m-1}]}c_{N,s}\int_{\Omega_k}\dfrac{h_k^{2\left(\frac{7}{8}\right)^m-\frac{7}{4}s}}{|e_N-\eta|^{N+2s}}\,d\eta\,d\mu(s)=0.
\end{equation}

%%Furthermore,\[
%%\begin{aligned}
%%\int_{(0,(7/8)^{m-1}]}&c_{N,s}\int_{\Omega_k}\dfrac{h_k^{2\left(\frac{7}{8}\right)^m-2s}}{|e_N-\eta|^{N+2s}}\,d\eta\,d\mu(s)\\ 
%%&= 
%%\int_{(0,(7/8)^{m}]}c_{N,s}\int_{\Omega_k}\dfrac{h_k^{2\left(\frac{7}{8}\right)^m-2s}}{|e_N-\eta|^{N+2s}}\,d\eta\,d\mu(s)+
%%\int_{((7/8)^{m},(7/8)^{m-1}]}c_{N,s}\int_{\Omega_k}\dfrac{h_k^{2\left(\frac{7}{8}\right)^m-2s}}{|e_N-\eta|^{N+2s}}\,d\eta\,d\mu(s).
%%\end{aligned}
%%\]
We observe that, for any $s\in ((7/8)^{m},(7/8)^{m-1}]$, 
\[
\lim_{k\to +\infty} c_{N,s}\int_{\Omega_k}\dfrac{h_k^{2\left(\frac{7}{8}\right)^m-2s}}{|e_N-\eta|^{N+2s}}\,d\eta=+\infty.
\]
Hence, by Fatou's lemma and~\eqref{I2mMM},
\begin{equation}\label{I2denominatore}
\liminf_{k\to+\infty}
\int_{((7/8)^{m},(7/8)^{m-1}]}c_{N,s}\int_{\Omega_k}\dfrac{h_k^{2\left(\frac{7}{8}\right)^m-2s}}{|e_N-\eta|^{N+2s}}\,d\eta\,d\mu(s)=+\infty.
\end{equation}
Using~\eqref{I_2numeratore} and~\eqref{I2denominatore} in
\eqref{I2minoreuguale7/8}, we obtain that
\begin{equation}\label{LIMI2}
\lim_{k\to+\infty} |I_2|=0.
\end{equation}
This and~\eqref{I1to0} imply~\eqref{pkconvergenza2}, as desired.

{F}rom~\eqref{pkconvergenza1} and~\eqref{pkconvergenza2} we
conclude that
\[
\lim_{k\to+\infty}u(p_k)=u(p),
\]
that is $u$ is continuous at $p$, as desired.
\end{proof}

Next result is a consequence of Theorem~\ref{propcontinuità}.

\begin{corollary}\label{corollariocontinuità}
Let $\Omega\subset \R^N$ be a domain with $C^1$ boundary and 
let $v_0\in C(\overline{\Omega})$. Let 
\[
v(x):=
\begin{cases}
v_0(x)    & \mbox{ if } x\in \overline{\Omega},
\\ 

\\
\dfrac{\displaystyle\int_{(0,1)} c_{N, s} \int_{\Omega}\dfrac{v_0(y)}{|x-y|^{N+2s}}\,dy\,d\mu(s)}{\displaystyle\int_{(0,1)}c_{N, s}\int_{\Omega}\dfrac{1}{|x-y|^{N+2s}}\,dy\,d\mu(s)}
& \mbox{ if } x\in \R^N\setminus \overline{\Omega}.
\end{cases}
\]
Then $v\in C(\R^N)$, it satisfies $v=v_0$ in $\overline{\Omega}$
and 
\[
\int_{(0,1)}\Ns u\,d\mu(s)=0 \quad\mbox{ in } \R^N\setminus \overline{\Omega}.
\]
\end{corollary}

\begin{proof}
By construction, $v=v_0$ in $\overline{\Omega}$
and $\int_{(0,1)}\Ns u\,d\mu(s)=0$ in $\R^N\setminus \overline{\Omega}$. Hence,  we can use Theorem~\ref{propcontinuità} to obtain that $v\in C(\R^N)$.
\end{proof}

We now study the behavior of the normalized Neumann function 
defined as
\begin{equation}\label{NTILDEE}
\widetilde{\mathscr N}u(x):=\frac{\displaystyle\int_{(0,1)}\Ns u(x)\,d\mu(s)}{\displaystyle\int_{(0,1)}c_{N, s}\int_{\Omega}\dfrac{1}{|x-y|^{N+2s}}\,dy\,d\mu(s)}.
\end{equation}

\begin{proposition}\label{ASMDNSDIKF}
Let $\Omega\subset \R^N$ be a domain with $C^1$ boundary and
let $u\in C(\R^N)$. Then,
\begin{equation}\label{limiteNtilde}
\lim_{\substack{x\to \partial \Omega \\ x\in \R^N\setminus \overline{\Omega}}}\widetilde{\mathscr N}u(x)=0.
\end{equation}
\end{proposition}

\begin{proof}
Let $\{x_k\}_k$ be a sequence in $\R^N\setminus \overline{\Omega}$
such that $x_k$ converges to $x_\infty\in \partial \Omega$ as $k$
goes to infinity.

{F}rom Corollary~\ref{corollariocontinuità} applied with $v_0:=u$,
there exists~$v\in C(\R^N)$ such that~$v=u$ in $\overline{\Omega}$
and~$\int_{(0,1)}\Ns v\,d\mu(s)=0$ in~$\R^N\setminus \overline{\Omega}$. By the continuity of~$u$ and~$v$ we also have
that 
\begin{equation}\label{limite5.14}
\lim_{k\to +\infty}\big(u(x_k)-v(x_k)\big)
=\big(u(x_\infty)-v(x_\infty)\big)=0.
\end{equation}
Moreover,
\[
\begin{aligned}
\widetilde{\mathscr N}u(x_k)&=\widetilde{\mathscr N}u(x_k)-\widetilde{\mathscr N}v(x_k)\\
&=\frac{\displaystyle\int_{(0,1)}c_{N, s}\int_{\Omega}\dfrac{u(x_k)-u(y)}{|x-y|^{N+2s}}\,dy\,d\mu(s)-\displaystyle\int_{(0,1)}c_{N, s}\int_{\Omega}\dfrac{v(x_k)-v(y)}{|x-y|^{N+2s}}\,dy\,d\mu(s)}{\displaystyle\int_{(0,1)}c_{N, s}\int_{\Omega}\dfrac{1}{|x-y|^{N+2s}}\,dy\,d\mu(s)}\\
&=\frac{\displaystyle\int_{(0,1)}c_{N, s}\int_{\Omega}\dfrac{u(x_k)-v(x_k)}{|x-y|^{N+2s}}\,dy\,d\mu(s)}{\displaystyle\int_{(0,1)}c_{N, s}\int_{\Omega}\dfrac{1}{|x-y|^{N+2s}}\,dy\,d\mu(s)}\\
&=u(x_k)-v(x_k).
\end{aligned}
\]
This and~\eqref{limite5.14} imply that
\[
\lim_{k\to +\infty}\widetilde{\mathscr N}u(x_k)=0,
\]
that is~\eqref{limiteNtilde}, which concludes the proof.
\end{proof}

\section{The superposition of fractional perimeters}\label{OSJDNDNMwefV}
In this section we discuss the relation between the superposition of nonlocal Neumann derivatives and the notion of superposition of fractional perimeters, which we introduce in here.

At first, we recall that, in the local case,  taking $\partial_\nu u =1$, the perimeter of $\Omega$
(often denoted by~$\mbox{Per}(\Omega)$) can be obtained as
\[
|\partial\Omega|=\int_{\partial\Omega} d\mathscr H^{N-1}_x = \int_{\partial\Omega} \partial_\nu u \, d\mathscr H^{N-1}_x.
\]
In the nonlocal case, discussed in \cite[Remark 3.4]{MR3651008}, a similar result was established
by considering a notion of normalized nonlocal derivative. We now extend this idea to cover the case of superpositions of operators (and perimeters) of different order.
To this end, one considers
\[w_{s, \Omega}(x):= c_{N, s} \int_\Omega \frac{dy}{|x-y|^{N+2s}}\]
and the normalized nonlocal derivative given by
\begin{equation}\label{NRSMCV}
\widetilde{\Ns}u(x):= \dfrac{\mathscr N_s u(x)}{w_{s, \Omega}(x)}\quad\mbox{ for } x\in\R^N\setminus\overline\Omega.
\end{equation}
In this setting, if $\widetilde{\Ns}u(x)=1$ for any $x\in\R^N\setminus\overline\Omega$, then the fractional perimeter, introduced in \cite{MR2675483} turns to be the integral of $\Ns u(x)$ over $\R^N\setminus\Omega$, namely
\begin{equation}\label{PERFRAC}
\mbox{Per}_s(\Omega) = \int_{\R^N\setminus\Omega} \Ns u(x) dx,
\end{equation}see~\cite[Remark 3.4]{MR3651008}.

In our setting, the normalized superposition of nonlocal derivatives is defined in \eqref{NTILDEE} and can be written as
\[
\widetilde{\mathscr N} u(x) = \dfrac{\displaystyle\int_{(0, 1)} \mathscr N_s \, u(x)\, d\mu(s)}{\displaystyle\int_{(0, 1)}w_{s, \Omega}(x)\, d\mu(s)}\quad\mbox{ for } x\in\R^N\setminus\overline\Omega.
\]
The reader can appreciate similarities and differences with respect to~\eqref{NRSMCV}.

Also, if $\widetilde{\mathscr N} u(x)=1$, the superposition of fractional perimeters, provided that $\mbox{Per}_s(\Omega)\in L^1((0, 1),  d\mu)$, can be obtained as
\begin{equation}\label{PERSUP}
\begin{split}
\int_{(0, 1)} \mbox{Per}_s(\Omega) \, d\mu(s)&:= \int_{(0, 1)} c_{N, s} \left(\int_{\Omega}\int_{\R^N\setminus\Omega} \dfrac{1}{|x-y|^{N+2s}} dx dy\right) d\mu(s)\\
&=\int_{(0, 1)} \left(\,\int_{\R^N\setminus\Omega} w_{s, \Omega}(x) \,dx\right) d\mu(s)\\
&=\int_{\R^N\setminus\Omega} \left(\,\int_{(0, 1)} w_{s, \Omega}(x) \, d\mu(s)\right) dx\\
&= \int_{\R^N\setminus\Omega} \left(\,\int_{(0, 1)} w_{s, \Omega}(x) \, d\mu(s)\right) \widetilde{\mathscr N}u(x)\, dx\\
&= \int_{\R^N\setminus\Omega} \left(\,\int_{(0, 1)} \mathscr N_s \, u(x) \, d\mu(s)\right)\, dx\\
&= \int_{(0, 1)} \left(\,\int_{\R^N\setminus\Omega} \mathscr N_s \, u(x) \, d\mu(s)\right)\, dx.
\end{split}
\end{equation}
This identity is interesting, because it recovers the superposition of fractional perimeters
from the superposition of fractional Neumann derivatives.

We stress that \eqref{PERSUP} can not be obtained by directly integrating the identity in \eqref{PERFRAC}.  This would be true if $\widetilde\Ns u(x) =1$ for any $s\in\mbox{supp}(\mu)$ which, in general, is not guaranteed by the condition $\widetilde{\mathscr N} u(x)=1$.

It is also worth noting that, if $\mu =\delta_s$ for some fractional power $s\in (0, 1)$, then \eqref{PERSUP} turns into \eqref{PERFRAC}.

\begin{appendix}

\section{Proof of Theorem \ref{uptoconstant}}\label{appendix1.3}
We now provide a complete proof of Theorem~\ref{uptoconstant}.

\begin{proof}[Proof of Theorem \ref{uptoconstant}]
We first deal with the case $g\equiv 0$ and $h\equiv 0$, corresponding to the homogeneous $(\alpha,\mu)$-Neumann conditions introduced in Definition~\ref{NDEFN}.
In this case, we assume that~$f\not \equiv 0$, otherwise the
constant is clearly the only solution, due to Lemma~\ref{lemmasoluzionecostante}.

Taking $\xi\in L^2(\Omega)$, we look for a function 
$v\in \mathcal{H}_{\alpha,\mu}(\Omega)$ which solves
\begin{equation}\label{problemaausiliariogh=0}
\begin{split}
\int_\Omega v\varphi\,dx&+\alpha \int_\Omega \nabla v\cdot\nabla \varphi\, dx\\
& +\int_{(0, 1)} \frac{c_{N, s}}{2} \iint_{\Q} \frac{(v(x)-v(y))(\varphi(x)-\varphi(y))}{|x-y|^{N+2s}} dx\,dy\, d\mu(s)
=\int_\Omega \xi\varphi\,dx,
\end{split}
\end{equation}
for any $\varphi\in \mathcal{H}_{\alpha,\mu}(\Omega)$, with
homogeneous $(\alpha,\mu)$-Neumann conditions.

Let us consider the functional 
$\mathcal{F}:\mathcal{H}_{\alpha,\mu}(\Omega) \to \R$ defined as
\[
\mathcal{F}(\varphi):=\int_\Omega \xi \varphi\,dx\quad \mbox{ for any } \varphi \in \mathcal{H}_{\alpha,\mu}(\Omega).
\]
Clearly, the functional $\mathcal{F}$ is linear, and it is also
continuous on $\mathcal{H}_{\alpha,\mu}(\Omega)$. Indeed, by Proposition~\ref{propcompactembedding}, we have
\[
|\mathcal{F}(\varphi)|\leq \int_\Omega |\xi|\,|\varphi|\,dx
\leq \|\xi\|_{L^2(\Omega)}\|\varphi\|_{L^2(\Omega)}
\leq \|\xi\|_{L^2(\Omega)}\|\varphi\|_{\alpha,\mu}.
\]
Thus, from the Riesz representation theorem we know that, for any given $\xi\in L^2(\Omega)$,  problem~\eqref{problemaausiliariogh=0} admits a unique solution $v\in \mathcal{H}_{\alpha,\mu}(\Omega)$.

Moreover, taking $\varphi:=v$ in~\eqref{problemaausiliariogh=0} and using again Proposition~\ref{propcompactembedding}, we obtain that
\begin{equation}\label{normaalfamu<normaL2}
\|v\|_{\alpha,\mu}=\int_\Omega \xi v\,dx
\leq C\|\xi\|_{L^2(\Omega)}.
\end{equation}
Now, we can define the operator~$T_0:L^2(\Omega)\to \mathcal{H}_{\alpha,\mu}(\Omega)$ as~$T_0 \xi:=v$,
and we also denote by $T$ its restriction operator in $\Omega$,
that is
\[
T \xi:=T_0 \xi\big|_{\Omega}.
\]
We remark that the function $T_0 \xi$ is defined in the whole of 
$\R^N$, while $T\xi$ is its restriction in $\Omega$. In light 
of this, we see that $T:L^2(\Omega) \to L^2(\Omega)$.

We claim that $T$ is a compact operator. To show this, we take a bounded 
sequence $\{\xi_k\}_{k\in \N}$ in $L^2(\Omega)$.
{F}rom~\eqref{normaalfamu<normaL2} we have that the sequence
$\{T_0\xi_k\}_{k\in \N}$ is bounded in $\mathcal{H}_{\alpha,\mu}(\Omega)$,
and from Proposition~\ref{propcompactembedding} we deduce that
the sequence $\{T\xi_k\}_{k\in \N}$ admits a subsequence that 
converges strongly in $L^2(\Omega)$. This proves that $T$ is compact.

Now, we claim that $T$ is a self-adjoint operator. To show this,
we take $\xi_1,\xi_2\in C^\infty_0(\Omega)$ and we use the weak
formulation in~\eqref{problemaausiliariogh=0}. {F}rom this, we have that for any $\varphi, \phi\in \mathcal{H}_{\alpha,\mu}(\Omega)$,
\begin{equation}\label{selfadjoint1}
\begin{aligned}
\int_\Omega T_0\xi_1&\varphi\,dx+\alpha \int_\Omega \nabla T_0\xi_1\cdot\nabla \varphi\, dx \\ 
&+\int_{(0, 1)} \frac{c_{N, s}}{2} \iint_{\Q} \frac{(T_0\xi_1(x)-T_0\xi_1(y))(\varphi(x)-\varphi(y))}{|x-y|^{N+2s}} dx\,dy\, d\mu(s) 
=\int_\Omega \xi_1\varphi\,dx
\end{aligned}
\end{equation}
and
\begin{equation}\label{selfadjoint2}
\begin{aligned}
\int_\Omega T_0\xi_2&\phi\,dx+\alpha \int_\Omega \nabla T_0\xi_2\cdot\nabla \phi\, dx \\ 
&+\int_{(0, 1)} \frac{c_{N, s}}{2} \iint_{\Q} \frac{(T_0\xi_2(x)-T_0\xi_2(y))(\phi(x)-\phi(y))}{|x-y|^{N+2s}} dx\,dy\, d\mu(s) 
=\int_\Omega \xi_2\phi\,dx.
\end{aligned}
\end{equation}
Taking $\varphi:=T_0\xi_2$ in~\eqref{selfadjoint1} and
$\phi:=T_0\xi_1$ in~\eqref{selfadjoint2}, we obtain that
\[
\int_\Omega \xi_1\, T_0\xi_2\,dx=\int_\Omega \xi_2\, T_0\xi_1\,dx\quad\mbox{ for any } \xi_1,\xi_2\in C^\infty_0(\Omega).
\]
Thus, since
$T_0\xi_1=T\xi_1$ and $T_0\xi_2=T\xi_2$ in $\Omega$, we 
conclude that 
\begin{equation}\label{selfadjoint3}
\int_\Omega \xi_1\, T\xi_2\,dx=\int_\Omega \xi_2\, T\xi_1\,dx\quad\mbox{ for any } \xi_1,\xi_2\in C^\infty_0(\Omega).
\end{equation}
If $\xi_1,\xi_2\in L^2(\Omega)$, then there exist two sequences~$\{\xi_{1,k}\}_{k\in \N}$ and~$\{\xi_{2,k}\}_{k\in \N}$ in~$C^\infty_0(\Omega)$ such that~$\xi_{1,k}$ converges to~$\xi_1$
and~$\xi_{2,k}$ converges to~$\xi_2$ in~$L^2(\Omega)$.
Then, from~\eqref{selfadjoint3}, we have that
\begin{equation}\label{selfadjoint4}
\int_\Omega \xi_{1,k}\, T\xi_{2,k}\,dx
=\int_\Omega \xi_{2,k}\, T\xi_{1,k}\,dx.
\end{equation}
Moreover, from~\eqref{normaalfamu<normaL2} we deduce that~$T\xi_{1,k}$ and~$T\xi_{2,k}$ converge respectively to~$T\xi_1$ and~$T\xi_2$ in~$L^2(\Omega)$. Hence,
\[
\lim_{k\to +\infty}\int_\Omega \xi_{1,k}\, T\xi_{2,k}\,dx
=\int_\Omega \xi_{1}\, T\xi_{2}\,dx
\]
and
\[
\lim_{k\to +\infty}\int_\Omega \xi_{2,k}\, T\xi_{1,k}\,dx
=\int_\Omega \xi_{2}\, T\xi_{1}\,dx.
\]
These facts and~\eqref{selfadjoint4} imply that 
\[
\int_\Omega \xi_{1}\, T\xi_{2}\,dx=\int_\Omega \xi_{2}\, T\xi_{1}\,dx \quad\mbox{ for any } \xi_1,\xi_2\in L^2 (\Omega).
\]
This proves that $T$ is a self-adjoint operator.

Now, we claim that
\begin{equation}\label{claimkercostanti}
Ker(Id-T) \mbox{ consists only of constant functions}.
\end{equation}
We first check that constant functions are in $Ker(Id-T)$. For this, let~$c\in\R$. We take a function costantly equal to~$c$ and observe that~$\Delta c=0=(-\Delta)^s c$ in~$\Omega$, and so~$L_{\alpha,\mu}(c)+c=c$. Moreover, we see that~$\partial_\nu c=0$ on~$\partial \Omega$ and 
\[
\int_{(0, 1)}\Ns c\,d\mu(s)=0
\]
in $\R^N\setminus \overline\Omega$. This shows that $T_0 c=c$ in $\R^N$,
and so $T c=c$ in $\Omega$, wich implies that $c\in Ker(Id-T)$.

Now, we show that if $\xi\in Ker(Id-T)\subseteq L^2(\Omega)$,
then $\xi$ is constant. By construction, 
$T_0 \xi\in \mathcal{H}_{\alpha,\mu}(\Omega)$ is a weak solution
of 
\begin{equation}\label{costanti1}
L_{\alpha,\mu}(T_0 \xi)+T_0 \xi =\xi\quad \mbox{ in }\Omega
\end{equation}
with
\begin{equation}\label{costanti2}
\partial_\nu (T_0 \xi)=0 \;\mbox{ on }\partial \Omega
\quad \mbox{and} \quad
\int_{(0, 1)}\Ns (T_0 \xi)\,d\mu(s)=0\; \mbox{ in }
\R^N\setminus \Omega.
\end{equation}
On the other hand, since $\xi \in Ker(Id-T)$ we have that
\begin{equation}\label{costanti3}
\xi =T\xi =T_0 \xi\quad \mbox{ in }\Omega.
\end{equation}
Hence, from~\eqref{costanti1}, $T_0 \xi$ is a weak solution of
\[
L_{\alpha,\mu}(T_0 \xi)=0 \quad\mbox{ in }\Omega.
\]
This, \eqref{costanti2} and Lemma~\ref{lemmasoluzionecostante}
imply that~$T_0 \xi$ is constant. Then, from~\eqref{costanti3} 
we have that~$\xi$ is constant in~$\Omega$, which concludes the 
proof of~\eqref{claimkercostanti}.

{F}rom~\eqref{claimkercostanti} and the Fredholm Alternative,
we have that
\[
Im(Id-T)=Ker(Id-T)^\perp=\{\mbox{constant functions}\}^\perp,
\]
where the orthogonality notion is in $L^2(\Omega)$, that is
\begin{equation}\label{immagineI-T}
Im(Id-T)=\left\{f\in L^2(\Omega):\,
\int_\Omega f\,dx =0 \right\}.
\end{equation}
Thus, taking $f$ such that $\int_\Omega f\,dx=0$, by~\eqref{immagineI-T} we know that there exists~$w\in L^2(\Omega)$
such that~$f=w-Tw$. Now, we set~$u:=T_0 w$. By construction,
$u$ is a weak solution of
\[
L_{\alpha,\mu}(T_0 w)+T_0 w= w \mbox{ in } \Omega,
\]
with $\partial_\nu u=0$ in 
$\partial \Omega$ and 
\[
\int_{(0, 1)}\Ns u\,d\mu(s)=0
\]
in $\R^N\setminus \overline\Omega$. As a consequence, 
\[
f=w-Tw=w-T_0 w= L_{\alpha,\mu}(T_0 w)=L_{\alpha,\mu}(u)\quad\mbox{ in } \Omega
\]
in the weak sense, so that $u$ is the desired solution.

Viceversa, let $u\in \mathcal{H}_{\alpha,\mu}(\Omega)$ be a weak 
solution of 
\[
\begin{cases}
L_{\alpha,\mu}(u)=f  \quad \mbox{ in } \Omega
\\
\mbox{with homogeneous } (\alpha,\mu)\mbox{-Neumann conditions}.
\end{cases}
\]
We set $w:=f+u$ and observe that
\[
L_{\alpha,\mu}(u)+u=f+u=w \quad\mbox{ in }\Omega
\]
in the weak sense. So, we have that $u=T_0 w$ in $\R^N$, and
hence $u=Tw$ in $\Omega$. Thus,
\[
(Id-T)w=w-u=f  \quad\mbox{ in }\Omega,
\]
and so $f\in Im(Id-T)$. Finally, from~\eqref{immagineI-T} we 
obtain that $\int_\Omega f\,dx=0$.

This conludes the proof in the case $g\equiv 0$ and $h\equiv 0$.

Now, we deal with the nonhomogeneous case~\eqref{prob1}.
By assumption, there exists a function~$\psi \in C^2(\R^N)$ 
such that~$\partial_\nu \psi =h$ on~$\partial \Omega$ and
\[
\int_{(0, 1)}\Ns \psi \,d\mu(s)=g\quad \mbox{ in } \R^N\setminus\overline{\Omega}.
\]
Thus, letting $\overline{u}=u-\psi$, we get that $\overline{u}$
is a weak solution of
\[
\begin{cases}
L_{\alpha,\mu}(\overline{u})=\overline{f}  \quad \mbox{in } \Omega
\\
\mbox{with homogeneous } (\alpha,\mu)\mbox{-Neumann conditions},
\end{cases}
\]
where~$\overline{f} :=f-L_{\alpha,\mu}(\psi)$.

Then, from the proof in the homogeneous case, this problem admits
a solution if and only if~$\int_\Omega \overline{f}\,dx=0$, 
that is 
\begin{equation}\label{eqnonomogeneo}
0=\int_\Omega \overline{f}\,dx
=\int_\Omega f\,dx -\int_\Omega  L_{\alpha,\mu}(\psi)\,dx.
\end{equation}
{F}rom Lemma~\ref{lemmathdivergenza} and the divergence theorem,
we have
\[
\begin{aligned}
\int_\Omega  L_{\alpha,\mu}(\psi)\,dx
&=-\alpha\int_\Omega \Delta \psi\,dx
+\int_{(0, 1)}\int_\Omega (-\Delta)^s \psi\,dx\,d\mu(s) \\
&=\alpha\int_{\partial\Omega} \partial_\nu  \psi\,d\mathscr{H}_x^{N-1}
+\int_{(0, 1)}\int_{\R^N\setminus \Omega }
\Ns \psi \,dx \,d\mu(s) \\
&=\alpha\int_{\partial\Omega} h\,d\mathscr{H}_x^{N-1}
+\int_{\R^N\setminus \Omega } g\,dx.
\end{aligned}
\]
{F}rom this and~\eqref{eqnonomogeneo} we conclude that a solution of~\eqref{prob1}
exists if and only if~\eqref{condizioneesistenzaunicità} holds.

Finally, thanks to Lemma~\ref{lemmasoluzionecostante} the solution 
is unique up to an additive constant, which completes the proof of the desired result.
\end{proof}

\section{Proof of \eqref{Hmuhilbert}}\label{appendixHilbert}
In this appendix we show that \eqref{Hmuhilbert} holds, namely $(\mathcal{H}_\mu(\Omega), \|\cdot\|_{\mu})$ is a Hilbert space.
The argument is rather standard, but it is included here for the reader's facility.

\begin{proof}[Proof of \eqref{Hmuhilbert}]
It is easy to check that~\eqref{prodottoscalaremu} is a bilinear form and~$\|u\|_\mu = (u, u)^{1/2}_\mu$. Moreover, if~$\|u\|_\mu = 0$, we get that~$\|u\|_{L^2(\Omega)}=0$. This gives~$u=0$ a.e. in~$\Omega$ and
\[
\int_{(0,1)} \left(c_{N,s}\iint_\Q \frac{|u(x)-u(y)|^2}{|x-y|^{N+2s}}\,dx\,dy\right)\,d\mu(s) =0,
\]
which, in turn, entails that
\[
\iint_\Q \frac{|u(x)-u(y)|^2}{|x-y|^{N+2s}}\,dx\,dy =0 \quad\mbox{ for any } s\in\mbox{supp}(\mu).
\]
Hence $|u(x)-u(y)|=0$ for any $(x, y)\in\Q$.  {F}rom this, we infer that for a.e.~$x\in \R^n\setminus\Omega$ and~$y\in\Omega$,
\[
u(x) = u(x) -u(y)=0,
\]
which means that $u=0$ for a.e. $x\in\R^N$. 

We now prove that $\mathcal{H}_\mu(\Omega)$ is complete. To this end,  let $u_k$ denote a Cauchy sequence with respect to the norm~\eqref{normamu}. Hence, $u_k$ is a Cauchy sequence in $L^2(\Omega)$ and converges, up to subsequences, to some $u\in L^2(\Omega)$ and also a.e. in $\Omega$. More precisely, this means that there exists a subset $Z_1\subset\R^N$ with\footnote{In~\eqref{Z1} we have used the standard notation~$|\cdot|$ to denote the Lebesgue measure of a set.}
\begin{equation}\label{Z1}
|Z_1|=0 \quad\mbox{ and }\quad u_k(x)\to u(x)\quad \mbox{ for all } x\in\Omega\setminus Z_1.
\end{equation}
Moreover, given any $U:\R^N\to\R$, for any $(x, y)\in\R^{2N}$ and $s\in (0, 1)$, we define the new function
\begin{equation}\label{Edefn}
E_U(x, y, s): = \frac{\big(U(x)-U(y)\big)\chi_{\Q}(x, y)}{|x-y|^{(N+2s)/2}}.
\end{equation}
Now, since
\[
E_{u_k}(x, y, s) - E_{u_h}(x, y, s) = \frac{\big(u_k(x) -u_k(y) -u_h(x) + u_h(y)\big)\chi_{\Q}(x, y)}{|x-y|^{(N+2s)/2}}
\]
and $u_k$ is a Cauchy sequence, for any $\varepsilon>0$ there exists~$N_\varepsilon\in\N$ such that, if~$h$, $k\ge N_\varepsilon$, 
\[
\varepsilon^2 \ge \int_{(0, 1)} c_{N, s} \iint_{\Q}\frac{|(u_k -u_h)(x) -(u_k -u_h)(y)|^2}{|x-y|^{N+2s}} \, dx \,dy\, d\mu (s) =: \|E_{u_k} - E_{u_h}\|^2_{L^2(\R^{2N}\times (0, 1))}.
\]
Thus, we have that $E_{u_k}$ is a Cauchy sequence in $L^2 \big(\R^{2N}\times (0, 1), dx\,dy\, d\mu\big)$. {F}rom this we infer that  $E_{u_k}$ converges, up to subsequences, to some $E$ in $L^2\big(\R^{2N}\times (0, 1), dx\,dy\, d\mu\big)$ and $E_{u_k}(x, y, s)$ converges to $E(x, y, s)$ a.e. in $\R^{2N}\times (0, 1)$.
Namely, there exist $Z_2\subset \R^{2N}$ and $\Sigma\subset (0, 1)$ such that
\begin{equation}\label{Z2}
\begin{split}
&|Z_2|=0, \quad \mu(\Sigma)=0 \\
&\mbox{and }\quad E_{u_k}(x, y, s) \to E(x, y, s) \mbox{ for all } (x, y)\in\R^{2N}\setminus Z_2, \ s\in (0, 1)\setminus\Sigma.
\end{split}
\end{equation}
Let fix $s\in (0, 1)\setminus\Sigma$. For any $x\in\Omega$, we set
\begin{align*}
S_x&:=\{y\in\R^N : (x, y)\in\R^{2N}\setminus Z_2\},\\
W&:= \{(x, y)\in\R^{2N} : x\in\Omega \mbox{ and } y\in\R^N\setminus S_x\}\\
{\mbox{and }}\quad V&:=\{x\in\Omega : |\R^N\setminus S_x|=0\}.
\end{align*}
We claim that
\begin{equation}\label{inZ2}
W\subseteq Z_2.
\end{equation}
Indeed, if $(x, y)\in W$, then $y\in\R^N\setminus S_x$, namely $(x, y)\in \R^{2N}\setminus Z_2$ and hence $(x, y)\in Z_2$, as desired. 

Accordingly, by~\eqref{Z2} and~\eqref{inZ2} we find that $|W|=0$. 

Hence, by Fubini's theorem it follows that
\[
0=|W|=\int_\Omega \left|\R^N\setminus S_x\right| dx,
\]
and thus $|\R^N\setminus S_x|=0$ for a.e. $x\in\Omega$. Also, we have $|\Omega\setminus V|=0$ which, together with~\eqref{Z1}, gives
\[
|\Omega\setminus (V\setminus Z_1)| = |(\Omega\setminus V)\cup Z_1| \le |\Omega\setminus V| +|Z_1| =0.
\]
In particular, we infer that $V\setminus Z_1$ is non empty. 

Let us fix $x_0\in V\setminus Z_1$. Now, since $x_0\in\Omega\setminus Z_1$, we have
\[
\lim_{k\to +\infty} u_k(x_0) = u(x_0)
\]
by~\eqref{Z1}. Moreover, $|\R^N\setminus S_{x_0}|=0$ since $x_0\in V$, namely for any $y\in S_{x_0}$ (i.e. for a.e. $y\in\R^N$), it follows that $(x_0, y)\in \R^{2N}\setminus Z_2$. Hence,  by~\eqref{Edefn} and~\eqref{Z2},
\begin{equation}\label{nons}
\lim_{k\to +\infty} E_{u_k}(x_0, y, s) = |x_0 -y|^{-(N+2s)/2} \lim_{k\to +\infty} \big( u_k(x_0) -u_k(y)\big) \chi_{\Q} (x_0, y) =  E(x_0, y, s).
\end{equation}
In addition, since $\Omega\times(\R^N\setminus\Omega)\subseteq \Q$, by the definition in~\eqref{Edefn},
\[
E_{u_k}(x_0, y, s):=\frac{u_k(x_0) -u_k(y)}{|x_0-y|^{(N+2s)/2}}\quad\mbox{ for a.e. } y\in\R^N\setminus \Omega.
\]
Hence, we have
\begin{align*}
\lim_{k\to +\infty} u_k(y) &= \lim_{k\to +\infty}\left(u_k(x_0) -|x_0-y|^{(N+2s)/2} E_{u_k}(x_0, y, s) \right)\\
&= u(x_0) -|x_0-y|^{(N+2s)/2} E(x_0, y, s)
\end{align*}
for a.e. $y\in\R^N\setminus\Omega$. 

Also, we can argue as in~\eqref{nons} to show that the latter limit does not depend on~$s$ as well. {F}rom this and~\eqref{Z1} we infer that~$u_k$ converges a.e. in $\R^N$. Up to a change of notation, we have that $u_k$ converges a.e. in $\R^N$ to some $u$. 

Now, recalling that $u_k$ is a Cauchy sequence in~$\mathcal{H}_\mu (\Omega)$,  fixed any $\varepsilon >0$, there exists~$N_\varepsilon\in\N$ such that, for any~$k\ge N_\varepsilon$, 
\begin{align*}
\varepsilon^2 &\ge \liminf_{h\to +\infty} \|u_k -u_h\|^2_{\mu}\\
&\ge \liminf_{h\to +\infty} \int_\Omega (u_k -u_h)^2 dx + \liminf_{h\to +\infty}\int_{\R^N\setminus\Omega} |g|(u_k -u_h)^2 dx\\
&\qquad +\frac12\liminf_{h\to +\infty} \int_{(0, 1)} c_{N, s} \int_{\Q}\frac{|(u_k -u_h)(x) - (u_k -u_h)(y)|^2}{|x-y|^{N+2s}} dx\,dy\, d\mu(s).
\end{align*}
Hence,  since $u_k$ converges to $u$ a.e. in $\R^N$,  by Fatou's Lemma we get
\begin{align*}
\varepsilon^2&\ge \int_\Omega (u_k -u)^2 dx + \int_{\R^N\setminus\Omega} |g|(u_k -u)^2 dx\\
&\qquad +\frac12 \int_{(0, 1)} c_{N, s} \int_{\Q}\frac{|(u_k -u)(x) - (u_k -u)(y)|^2}{|x-y|^{N+2s}} dx\,dy\, d\mu(s)\\
&= \|u_k -u\|^2_\mu,
\end{align*}
namely,  $u_k$ converges to $u$ in $\mathcal{H}_\mu (\Omega)$, i.e. $\mathcal{H}_\mu (\Omega)$ is complete.
\end{proof}

\section{Proof of Proposition \ref{propfunctional}}\label{appendixfunctional}
In this appendix we prove that Proposition \ref{propfunctional} holds. The argument presented is rather standard, but it is included here for the reader's facility.

\begin{proof}[Proof of Proposition \ref{propfunctional}]
Let $u\in \mathcal{H}_{\alpha, \mu}(\Omega)$. By the H\"older inequality and recalling definition~\eqref{normaalfamu},
\[
\left| \int_\Omega fu \, dx \right|\le \|f\|_{L^2(\Omega)}\|u\|_{L^2(\Omega)}\le c_1 \|u\|_{\alpha, \mu}
\]
and
\[
\left|\; \int_{\R^N\setminus \Omega} gu \, dx \right|\le \int_{\R^N\setminus\Omega} |g|^{1/2} |g|^{1/2} |u|\le \|g\|^{1/2}_{L^1(\R^N\setminus\Omega)} \| |g|^{1/2} u\|_{L^2(\R^N\setminus\Omega)}\le c_2 \|u\|_{\alpha, \mu},
\]
for some positive constants $c_1, c_2$. 

Similarly, there exists $c_3>0$ such that
\[
\left| \,\int_{\partial\Omega} hu \, d\mathscr{H}_x^{N-1} \right|\le \|h\|^{1/2}_{L^1(\partial\Omega)} \, \| |h|^{1/2} u\|_{L^2(\partial\Omega)} \le c_3 \|u\|_{\alpha, \mu}.
\]
{F}rom this, we infer that if $u\in \mathcal{H}_{\alpha, \mu}(\Omega)$,
\[
| I(u)|\le c_4\|u\|_{\alpha, \mu} <+\infty.
\]

We now evaluate the first variation of the functional $I$. To this end, let~$\varepsilon\in (0, 1)$ and take~$v\in \mathcal{H}_{\alpha, \mu}(\Omega)$. Hence,
\begin{align*}
I(u+\varepsilon v) &= \frac{\alpha}{2}\int_\Omega |\nabla (u+\varepsilon v)|^2 dx + \int_{(0, 1)} \frac{c_{N, s}}{4} \iint_{\Q} \frac{|(u+\varepsilon v)(x)-(u+\varepsilon v)(y)|^2}{|x-y|^{N+2s}} dx\,dy\, d\mu(s)\\
&\quad - \int_\Omega f \, (u+\varepsilon v) \, dx -\int_{\R^N\setminus\Omega} g \, (u+\varepsilon v) \, dx -\int_{\partial\Omega} h \, (u+\varepsilon v) \, d\mathscr{H}_x^{N-1}\\
&= I(u) +\varepsilon\Bigg(\alpha\int_\Omega \nabla u\cdot\nabla v \, dx +\int_{(0, 1)}\frac{c_{N, s}}{2} \iint_{\Q} \frac{(u(x)-u(y)) (v(x)-v(y))}{|x-y|^{N+2s}} dx\,dy\, d\mu(s)\\
&\qquad\qquad -\int_{\Omega} f\, v \, dx -\int_{\R^N\setminus\Omega} g \, v \, dx -\int_{\partial\Omega} h \, v \, d\mathscr{H}_x^{N-1} \Bigg)\\
&\qquad\qquad+\varepsilon^2 \left(\frac{\alpha}{2}\int_{\Omega} |\nabla v|^2 dx + \int_{(0, 1)}\frac{c_{N, s}}{4} \iint_{\Q}\frac{|v(x)-v(y)|^2}{|x-y|^{N+2s}} dx\,dy\, d\mu(s)\right).
\end{align*}
{F}rom this, we infer that
\begin{align*}
\lim_{\varepsilon\to 0} \frac{I(u+\varepsilon v) - I(u)}{\varepsilon} &= \alpha\int_\Omega \nabla u\cdot\nabla v \, dx +\int_{(0, 1)}\frac{c_{N, s}}{2} \iint_{\Q} \frac{(u(x)-u(y)) (v(x)-v(y))}{|x-y|^{N+2s}} dx\,dy\, d\mu(s)\\
& \qquad-\int_{\Omega} fv \, dx -\int_{\R^N\setminus\Omega} gv \, dx -\int_{\partial\Omega} hv \, d\mathscr{H}_x^{N-1},
\end{align*}
namely
\begin{align*}
(I'(u), v) & = \alpha\int_\Omega \nabla u\cdot\nabla v \, dx +\int_{(0, 1)}\frac{c_{N, s}}{2} \iint_{\Q} \frac{(u(x)-u(y)) (v(x)-v(y))}{|x-y|^{N+2s}} dx\,dy\, d\mu(s)\\
&\qquad -\int_{\Omega} fv \, dx -\int_{\R^N\setminus\Omega} gv \, dx -\int_{\partial\Omega} hv \, d\mathscr{H}_x^{N-1}.
\end{align*}
Hence, if $u$ is a critical point of $I$, we have that $u$ is a weak solution of~\eqref{prob1}, according to the definition stated in~\eqref{weaksolution}. This completes the proof.
\end{proof}

\section{Proof of Proposition \ref{ASUYY}}\label{appendixASUYY}
Here we prove Proposition \ref{ASUYY}.

\begin{proof}[Proof of Proposition \ref{ASUYY}]
Let 
\[
m:=\frac{1}{|\Omega|}\int_\Omega u_0(x)\,dx
\]
be the total mass of $u$, which does not depend on $t$ thanks to 
Proposition~\ref{proposizionemassacostante}.
Defining 
\[
A(t):=\int_\Omega |u(x,t)-m|^2\,dx,
\]
from Proposition~\ref{proposizionemassacostante} we have
\[
A(t)=\int_\Omega (u^2(x,t)+2mu(x,t)+m^2)\,dx=
\int_\Omega u^2(x)\,dx -|\Omega|m^2.
\]
In light of~\eqref{ipotesimassa1}, \eqref{ipotesimassa2},
\eqref{ipotesimassa3} and~\eqref{ipotesimassa4}, we can use 
Lemma~\ref{lemmaperparti}, the classical integration by parts formula and the homogeneous Neumann conditions to obtain
\[
\begin{aligned}
A'(t)&=2\int_\Omega \partial_t u(x,t)u(x,t)\,dx
=2 \alpha\int_\Omega u(x,t) \Delta u(x,t)\,dx
-2 \int_\Omega u(x,t) (-\Delta)^s u(x,t)\,dx \\
&=-2 \alpha\int_\Omega |\nabla u(x,t)|^2 \,dx
-\int_{(0,1)}c_{N,s}
\iint_\Q \frac{|u(x,t)-u(y,t)|^2}{|x-y|^{N+2s}}\,dx\,dy\,d\mu(s)\\
&<0,
\end{aligned}
\]
which gives that $A$ is decreasing.

Moreover, if $\mu\not \equiv 0$, we can use the Poincar\'e
inequality in Lemma~\ref{POINCARE} and Proposition~\ref{proposizionemassacostante} to get
\begin{equation}\label{equazioneA'<A}\begin{split}&
A'(t) \leq -\int_{(0,1)}c_{N,s}
\iint_\Q \frac{|u(x,t)-u(y,t)|^2}{|x-y|^{N+2s}}\,dx\,dy\,d\mu(s)\\&\qquad\qquad
\leq -c\int_\Omega |u(x,t)-m|^2\,dx=-cA(t),\end{split}
\end{equation}
for some $c>0$. We notice that, in the case $\mu\equiv 0$,
one can prove a Poincar\'e inequality as in Lemma~\ref{POINCARE} 
and still obtain the inequality in~\eqref{equazioneA'<A}.

Thus, from~\eqref{equazioneA'<A} it follows that~$
A(t)\leq e^{-ct}A(0)$,
and so
\[
\lim_{t\to +\infty} \int_\Omega |u(x,t)-m|^2\,dx=0,
\]
that is, $u$ converges to $m$ in $L^2(\Omega)$ as $t\to +\infty$.
\end{proof}

\section{An alternative proof of~\eqref{LIMI2}}\label{KDMCLDFMVV4}
Here we provide an alternative proof of~\eqref{LIMI2} in the proof of Theorem~\ref{propcontinuità}.

\begin{lemma}
Under the same assumptions of Theorem~\ref{propcontinuità},  recalling the definition in~\eqref{I2}, we have
\[
\lim_{k\to+\infty} |I_2| =0.
\]
\end{lemma}

\begin{proof}
The same arguments as in Theorem~\ref{propcontinuità} allow us to state inequality~\eqref{I2minoreuguale}, namely
\[
|I_2|\leq 2\|u\|_{L^\infty(\overline{\Omega})}
\dfrac{\displaystyle\int_{(0,1)}c_{N,s}\int_{\Omega_k}\dfrac{h_k^{-\frac{7}{4}s}}{|e_N-\eta|^{N+2s}}\,d\eta\,d\mu(s)}{\displaystyle\int_{(0,1)}c_{N,s}\int_{\Omega_k}\dfrac{h_k^{-2s}}{|e_N-\eta|^{N+2s}}\,d\eta\,d\mu(s)}.
\]
Let us define
\begin{align}\label{alphaks}
\alpha_{k, s}&:= h_k^{-\frac74 s} c_{N, s} \int_{\Omega_k}\dfrac{d\eta}{|e_N-\eta|^{N+s}}
\\ {\mbox{and }}\quad
\label{betalphaks}
\beta_{k, s}&:= h_k^{-2s} c_{N, s} \int_{\Omega_k}\dfrac{d\eta}{|e_N-\eta|^{N+2s}}.
\end{align}

We claim that 
\begin{equation}\label{claimappendix}
\lim_{k\to +\infty} \dfrac{\displaystyle\int_{(0,1)} \alpha_{k, s} \, d\mu(s)}{\displaystyle\int_{(0,1)}\beta_{k, s} \, d\mu(s)} = 0.
\end{equation}
First, we notice that for $k$ large enough, it holds that~$
\Omega_k \subset \R^N\setminus B_{1/2}(e_N)$.
{F}rom this, we have
\begin{equation}\label{alphaksminore}
\begin{split}&
\alpha_{k, s}\le h_k^{-\frac74 s} c_{N, s} \int_{\R^N\setminus B_{1/2}(e_N)}\dfrac{d\eta}{|e_N-\eta|^{N+s}} = h_k^{-\frac74 s} c_{N, s} \int_{\R^N\setminus B_{1/2}} \dfrac{dz}{|z|^{N+s}}\\
&\qquad\qquad= h_k^{-\frac74 s} c_{N, s} \omega_{N-1} \int_{1/2}^{+\infty} \rho^{-1-s} d\rho= \frac{2^s}{s} \omega_{N-1}\, c_{N, s}\, h_k^{-\frac74 s}.
\end{split}
\end{equation}
We define the set
\[
\mathcal C:=\left\{x\in\R^N: x_N\le 1-\frac{9}{10}\sqrt{x_1^2 +\dots+x^2_{N-1}} \,, \ 2\le |x-1|\le \frac{1}{2\sqrt{h_k}} \right\},
\]
as shown in Figure~\ref{FIG2}. Since $\Omega_k$ converges to the halfplane~$\Pi$, we get
\[
\mathcal C \subset \left(\Omega_k\cap B_{1/2\sqrt{h_k}}\right) \subset\Omega_k\quad\mbox{ for sufficiently large } k.
\]

\begin{figure}[h]
\begin{center}
\includegraphics[scale=.42]{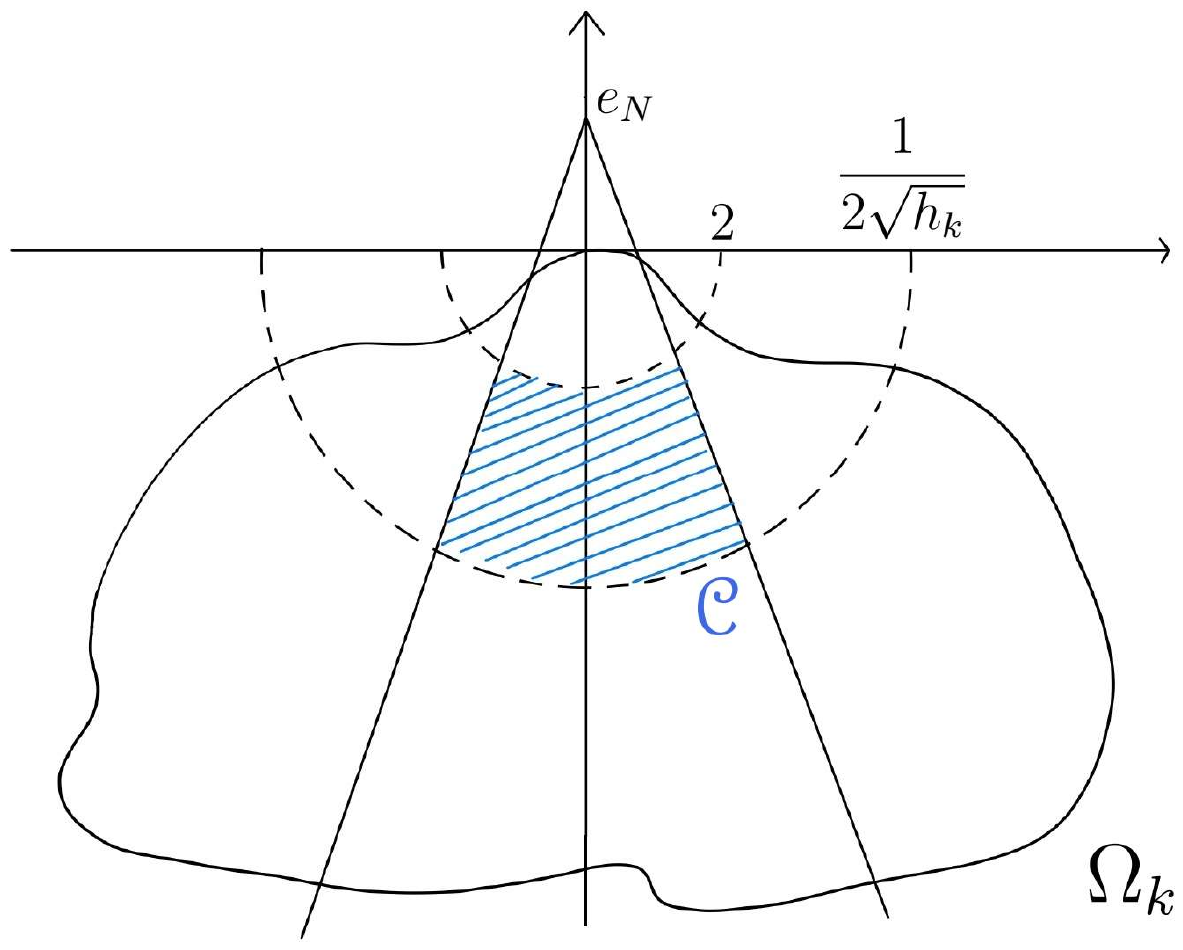}
\end{center}
\caption{The set $\mathcal C$ (in blue).}
\label{FIG2}
\end{figure}

Moreover, we can write
\[
\mathcal C=\left\{(\rho, \theta_1,\dots, \theta_{N-1}) \in\R\times [0, 2\pi]^{N-2}\times [0, \pi] : \ \rho\in \left[2, \frac{1}{2\sqrt{h_k}}\right],  \ \theta=(\theta_1,\dots, \theta_{N-1})\in\Sigma \right\},
\]
with $\Sigma\subseteq [0, 2\pi]^{N-2}\times [0, \pi]$. 

{F}rom this and recalling the definition in~\eqref{betalphaks}, we have
\begin{equation}\label{betaksmaggiore}
\begin{split}
\beta_{k, s}&\ge h_k^{-2s} c_{N, s} \int_{\mathcal C} \dfrac{d\eta}{|e_N-\eta|^{N+2s}} = h_k^{-2s} c_{N, s} \int_2^{\frac{1}{2\sqrt{h_k}}} \left(\int_{\Sigma} \dfrac{\rho^{N-1}}{\rho^{N+2s}} d\theta\right) d\rho\\
&\quad=h_k^{-2s} c_{N, s} \left(\int_{\Sigma} d\theta\right) \int_2^{\frac{1}{2\sqrt{h_k}}} \rho^{-1-2s} d\rho = h_k^{-2s} c_{N, s} \, \dfrac{\omega_{N-1}^{\mathcal C}}{2s} [2^{-2s} -2^{2s} h_k^s],
\end{split}
\end{equation}
where we introduced the notation
\[
\omega_{N-1}^{\mathcal C}:= \int_{\Sigma} d\theta.
\]
Now, by~\eqref{alphaksminore} and~\eqref{betaksmaggiore} we infer that
\[
\dfrac{\alpha_{k, s}}{\beta_{k, s}} \le \dfrac{\frac{2^s}{s} \omega_{N-1} \, c_{N, s} \, h_k^{-\frac74 s}}{h_k^{-2s} \, c_{N, s} \, \dfrac{\omega_{N-1}^{\mathcal C}}{2s} \, [2^{-2s} -2^{2s} h_k^s]}= \dfrac{2^{s+1} \, \omega_{N-1} \, h_k^{\frac{s}{4}}}{\omega_{N-1}^{\mathcal C}-h_k^s \, \omega_{N-1}^{\mathcal C} 2^{2s}}.
\]
Let us take $\delta>0$ such that~$
\mu\big([\delta, 1)\big)>0$.
We notice that
\[
\dfrac{\alpha_{k, s}}{\beta_{k, s}} \le C h_k^{\frac{\delta}{4}} \qquad\mbox{ for } s\in (\delta, 1),
\]
for a positive constant $C>0$ which does not depend on $s$.
Hence, we have 
\[
\dfrac{\displaystyle\int_{(\delta,1)} \alpha_{k, s} \, d\mu(s)}{\displaystyle\int_{(0,1)}\beta_{k, s} \, d\mu(s)}\le C h_k^{\frac{\delta}{4}} \ \dfrac{\displaystyle\int_{(\delta,1)} \beta_{k, s} \, d\mu(s)}{\displaystyle\int_{(0,1)}\beta_{k, s} \, d\mu(s)}\le C h_k^{\frac{\delta}{4}}.
\]
In light of the previous inequality and recalling that $h_k\to 0$ a $k\to +\infty$, we get
\begin{equation}\label{quasi}
\lim_{k\to +\infty} \dfrac{\displaystyle\int_{(\delta,1)} \alpha_{k, s} \, d\mu(s)}{\displaystyle\int_{(0,1)}\beta_{k, s} \, d\mu(s)} = 0.
\end{equation}
Besides,  multiplying and dividing by $h_k^{\frac74 \delta}$, we have
\[
\dfrac{\displaystyle\int_{(0, \delta)} \alpha_{k, s} \, d\mu(s)}{\displaystyle\int_{(0,1)}\beta_{k, s} \, d\mu(s)}\le \dfrac{\omega_{N-1}\displaystyle\int_{(0, \delta)} \frac{2^s}{s} \, c_{N, s}\, h_k^{\frac74 (\delta - s)} \, d\mu(s)}{\displaystyle\int_{(0,\delta)}\beta_{k, s} \, d\mu(s) +\omega_{N-1}^{\mathcal C}\displaystyle\int_{(\delta, 1)} \, \dfrac{c_{N, s}}{2s} \left[2^{-2s} h_k^{-2s+\frac74 \delta} -2^{2s} h_k^{-s +\frac74 \delta}\right] \, d\mu(s)}.
\]
We notice that, by the dominated convergence theorem, it follows that
\[
\lim_{k\to +\infty} \displaystyle\int_{(0, \delta)} \frac{2^s}{s} \, c_{N, s}\, h_k^{\frac74 (\delta - s)} \, d\mu(s) =0.
\]
Moreover,  since $-2s +\frac74 \delta<0$ for $s\in (\delta, 1)$, we get
\[
\lim_{k\to +\infty} \int_{(\delta, 1)} \, \dfrac{c_{N, s}}{2s} \left[2^{-2s} h_k^{-2s+\frac74 \delta} -2^{2s} h_k^{-s +\frac74 \delta}\right] \, d\mu(s) =+\infty.
\]
Therefore, we infer that
\begin{equation}\label{quasi2}
\lim_{k\to +\infty} \dfrac{\displaystyle\int_{(0, \delta)} \alpha_{k, s} \, d\mu(s)}{\displaystyle\int_{(0,1)}\beta_{k, s} \, d\mu(s)} =0.
\end{equation}
Finally, by~\eqref{quasi} and~\eqref{quasi2} we have that~\eqref{claimappendix} holds. This concludes the proof.
\end{proof}

\end{appendix}

\section*{Acknowledgements} 
All the authors are members of the Australian Mathematical Society (AustMS). CS and EPL are members of the INdAM--GNAMPA.

This work has been supported by the 
Australian Future Fellowship FT230100333 and by the
Australian Laureate Fellowship FL190100081.

%\newpage
{\small \def\cdprime{$''$}

\vfill

\end{document}